\documentclass[11pt]{amsart}

\usepackage{amssymb,amsmath,enumerate,epsfig}
\usepackage{amsfonts,color}
\usepackage{a4,latexsym}
\usepackage{parskip}

\usepackage{hyperref}
\hypersetup{pdftitle=5}

\newcommand{\C} {\mathbb{C}}
\newcommand{\Q} {\mathbb{Q}}
\newcommand{\N}  {\mathbb{N}}
\newcommand{\R} {\mathbb{R}}
\newcommand{\F}{\mathbb{F}}
\newcommand{\Z}{\mathbb{Z}}
\newcommand{\p}{\mathfrak{p}}

\newcommand{\OO}{\mathcal{O}}

\newcommand{\PP}{\mathbb{P}}
\newcommand{\id}{\mathop{\rm id}}
\newcommand{\rank}{\mathop{\rm rank}}
\newcommand{\Pic}{\mathop{\rm Pic}}
\newcommand{\NS}{\mathop{\rm NS}}

\newcommand{\MWL}{\mathop{\rm MWL}}

\newcommand{\disc}{\mathop{\rm disc}}
\newcommand{\A}{\mathfrak{A}}

\newcommand{\D}{\mathfrak{D}}

\newcommand{\Aut}{\mathop{\rm Aut}}
\newcommand{\Gal}{\mathop{\rm Gal}}

\newcommand{\mX}{\mathcal X}

\newcommand{\mS}{\mathcal S}
\newcommand{\mY}{\mathcal Y}

\newcommand{\het}[1]{H_\text{\'et}^2(#1,\Q_\ell)}
\newcommand{\tr}{\mathop {\rm tr}}
\newcommand{\Fr}{\mathop {\rm Frob}}
\newcommand{\Hp}{H^2_\text{prim}(S_m)}

\newtheorem{Theorem}{Theorem}[section]
\newtheorem{Proposition}[Theorem]{Proposition}
\newtheorem{Lemma}[Theorem]{Lemma}

\newtheorem{Corollary}[Theorem]{Corollary}

\theoremstyle{remark}
\newtheorem{Remark}[Theorem]{Remark}

\newtheorem{Technique}[Theorem]{Technique}
\newtheorem{Example}[Theorem]{Example}

\theoremstyle{definition}

\newtheorem{Problem}[Theorem]{Problem}

\numberwithin{equation}{section}

\begin{document}

\title{Picard numbers of quintic surfaces}



\author{Matthias Sch\"utt}
\address{Institut f\"ur Algebraische Geometrie, Leibniz Universit\"at
  Hannover, Welfengarten 1, 30167 Hannover, Germany}
\email{schuett@math.uni-hannover.de}
\urladdr{http://www.iag.uni-hannover.de/schuett/}

\thanks{Funding   by ERC StG~279723 (SURFARI)
 is gratefully acknowledged}

\date{September 21, 2014}

 \begin{abstract}
We solve the Picard number problem for complex quintic surfaces
by proving that every number between 1 and 45 occurs as Picard number of a quintic surface over $\Q$.
Our main technique consists in arithmetic deformations of Delsarte surfaces,
but we also use K3 surfaces and wild automorphisms.
 \end{abstract}
 
 \maketitle


%

\section{Introduction}

The Picard number of an algebraic surfaces roughly measures the quantity of curves on the surface,
or rather how complicated these curves become.
Arguably it is the most important invariant of an algebraic surface that is not preserved under deformations.
Yet the Picard number is notoriously hard to compute 
despite recent progress in several directions (see Section \ref{s:pic}).
This brings us to the following fundamental problem:

\begin{Problem}
\label{prob}
Which Picard numbers occur in a given deformation class of algebraic surfaces?
\end{Problem}

Here we will mostly be concerned with complex algebraic surfaces,
but Problem \ref{prob} makes perfect sense over any field,
and in particular over $\Q$.
In fact, we will always try to exhibit surface models over $\Q$,
especially for the reason that we want to make  use of reductions modulo primes
while circumventing ramification (see Theorem \ref{thm:prim}).

The answer to Problem \ref{prob} is known for several classes of complex surfaces,
for instance, for abelian surfaces, K3 surfaces, elliptic surfaces (over $\PP^1$  with section),
and generally for surfaces with vanishing geometric genus.
This paper addresses the first unknown case,
namely quintic surfaces in $\PP^3$.
For deformation reasons, we are led to allow the quintics to admit isolated 
rational double point singularities; in this case,
we consider the minimal desingularisation instead.
Our main result is the following.

\begin{Theorem}
\label{thm}
Complex quintic surfaces attain all possible Picard numbers,
i.e. from $1$ to $h^{1,1}=45$.
\end{Theorem}

More precisely, we will prove
that each Picard number is realised by a quintic over $\Q$.
One might wonder whether an analytic approach using higher Noether-Lefschetz loci
in the moduli space would yield the existence result of Theorem \ref{thm}.
For K3 surfaces, for instance, the problem can be solved over $\C$ by translating it into lattice theory (see \ref{ss:K3}).
For surfaces of general type, however, it seems non-trivial to decide
whether the higher Noether-Lefschetz loci are non-empty, especially for high Picard numbers,
and even if so, whether there exists a closed point in the moduli space of the expected Picard number.

The proof of Theorem \ref{thm} rests on the results on Picard numbers of Delsarte surfaces from \cite{S-quintic}.
We introduce a novel technique based on arithmetic deformations of Delsarte surfaces
which is developed in detail in Section \ref{s:af} (see Technique \ref{tech-1}).
Combined with Galois actions and automorphisms (Techniques \ref{tech-2}, \ref{tech-3}),
this will enable us to exhibit explicit quintics with given Picard number 
for almost every value given in Theorem \ref{thm}.
There are only 4 Picard numbers for which we rely on different methods, see Section \ref{s:other}.
As indicated above, all of these quintics have a model over $\Q$. 

For reference, we highlight a distinct feature which we find special about the arithmetic deformation technique
and our explicit examples:

\begin{itemize}
\item
For small Picard number, it is usually feasible to find candidate surfaces,
but rather non-trivial to verify that the Picard number is as expected.
\item
For large Picard numbers, it is often hard to find candidate surfaces in the first instance
while the verification of the Picard number tends to be easier.
\end{itemize}

Both problems are solved  for quintic surfaces in the course of the proof of Theorem \ref{thm}.

The arithmetic deformation technique itself may also be applied to other classes of surfaces.
As an illustration, we will derive in a systematical way a plentitude 
of K3 surfaces of Picard number $1$, see \ref{ss:K3-1}.
In fact, K3 surfaces also enter the proof of Theorem \ref{thm}
at several steps, often critically,
once even combined with the action of wild automorphisms following \cite{S-wild}
(see Section \ref{s:6-8}).

The paper is organised as follows.
In the next section we review what seems to be known about Picard numbers of algebraic surfaces. 
Section \ref{s:del} is devoted to Delsarte surfaces and their Picard numbers.
Arithmetic deformations are developed in Section \ref{s:af}.
As a first application, we exhibit a plentitude of K3 surfaces of Picard number $1$
(joint with Ronald van Luijk).
For quintic surfaces, we work out obstructions against lifting of divisor classes in Section \ref{s:lift}.
The subsequent sections present the quintics needed to prove Theorem \ref{thm}.


\section{Picard numbers of algebraic surfaces}
\label{s:pic}

Throughout this paper, an algebraic surface $S$ means
a minimal smooth projective surface (thought geometrically, i.e.~over an algebraically closed field).
In particular, if we are given a birational model with isolated singularities,
then the corresponding surface $S$ is understood to be the minimal desingularisation.

To understand the inner structure of an algebraic surface $S$,
and in particular its cohomology,
it is indispensable to understand the curves contained in $S$,
and the divisors formed by them.
Following the case of algebraic curves,
one is led to introduce the notion of linear equivalence $\sim$ on the formal group $\mbox{Div}(S)$ of divisors on $S$.
The quotient is the \emph{Picard group}
\[
\Pic(S) = \mbox{Div}(S)/\sim.
\]
If $S$ does not have irregularity $q(S)=0$,
then $\Pic(S)$ contains a continuous part
accounting for divisors moving in families on $S$.
In any case,
this can be killed by algebraic equivalence $\approx$,
reducing, so to say, to the discrete part of $\Pic(S)$,
the so-called \emph{N\'eron-Severi group}
\[
\NS(S) = \mbox{Div}(S)/\approx.
\]
By the theorem of the base,
$\NS(S)$ always is a finitely generated abelian group.
Its rank is called the \emph{Picard number} and denoted by $\rho(S)$:
\[
\rho(S) = \mbox{rank} \NS(S).
\]
{\bf Disclaimer:}
Throughout this paper, we are only concerned with geometric Picard numbers.
That is, even if the surface $S$ is given by equations over some non-closed field $k$,
we compute $\rho(S)$ for the base extension of $S$ to an algebraic closure of $k$.

\subsection{}

In practice, in order to work out the Picard number,
it sometimes suffices to compute intersection numbers of divisors,
since the corresponding notion of numerical equivalence $\equiv$
only kills the torsion in $\NS(S)$.
That is, if we manage to find divisors on $S$
whose intersection matrix has rank $r$,
then we infer that $\rho(S)\geq r$.
Implicitly, we see here that $\NS(S)$ is equipped with a quadratic form
which is compatible with cup-product on $H^2(S)$ via the cycle class map, see \ref{ss:2.2}.
By the Hodge index theorem, 
this quadratic form is non-degenerate of signature $(1,\rho(S)-1)$ on $\NS(S)\otimes\R$.

We remark that for many important classes of algebraic surfaces,
such as smooth surfaces in $\PP^3$ 
(or if there are only isolated rational double points as singularities, the minimal resolution), 
K3 surfaces or elliptic surfaces with section over $\PP^1$,
all three notions of linear, algebraic and numerical equivalence coincide.
For our quintics we will thus use $\Pic(S)$ and $\NS(S)$ interchangeably
and also refer to these groups as lattices with the intersection form.

\subsection{}
\label{ss:2.2}

Having seen that intersection numbers give rise to lower bounds for the Picard number,
it remains to discuss possible ways to derive upper bounds.
Here cohomology enters the game through the cycle class map
which embeds $\NS(S)\otimes\Q$ into $H^2(S,\Q)$ for complex surfaces,
or generally
\begin{eqnarray}
\label{eq:et}
\NS(S)\otimes\Q_\ell\hookrightarrow H^2_\text{\'et}(S,\Q_\ell(1))
\end{eqnarray}
for $\ell$-adic \'etale cohomology with a Tate twist.
The latter embedding directly gives the characteristic free upper bound
\begin{eqnarray}
\label{eq:Igusa}
\rho(S)\leq b_2(S)
\end{eqnarray}
originally due to Igusa \cite{Igusa}.
Over $\C$, this can be improved thanks to the Hodge decomposition and Lefschetz' theorem
which states that
\[
(\NS(S)/\text{torsion}) \cong H^2(S,\Z) \cap H^{1,1}(S).
\]
In particular, this gives
\begin{eqnarray}
\label{eq:Lef}
\rho(S) \leq h^{1,1}(S),
\end{eqnarray}
the best general upper bound for the Picard number over $\C$.
As an application, we read off for a surface with geometric genus $p_g(S)=0$,
either complex or lifting to characteristic zero,
that $\rho(S)=b_2(S)$.

For later use, we mention the relevant invariants for K3 surfaces and quintics
that can be obtained from classical formulae such as Noether's formula:
\begin{itemize}
\item
a K3 surface $S$ has $b_2(S)=22, p_g(S)=1, h^{1,1}(S)=20$;
\item
a quintic $S$ has $b_2(S)=53, p_g(S)=4, h^{1,1}(S)=45$.
\end{itemize}

\subsection{Singular and supersingular surfaces}

In view of the upper bounds \eqref{eq:Lef} resp.~\eqref{eq:Igusa},
given an algebraic surface $S$,
there is one Picard number with an obvious strategy to prove:
the maximum one where one 'only' has to exhibit enough independent divisors on $S$.
Surfaces attaining the maximum of \eqref{eq:Lef} over $\C$ are sometimes called singular (in the sense of exceptional),
and those with $\rho(S)=b_2(S)$ (over fields of positive characteristic) are often referred to as supersingular.
It may come as a bit surprising at first that the latter turn out to be much easier to find.
For instance, any unirational surface is automatically supersingular,
and in fact one can explicitly determine the characteristics where a Fermat surface of given degree
becomes supersingular as we shall exploit in the next section (see Example \ref{ex:TT}).
In contrast, singular surfaces tend to be much harder to find, and it was only in \cite{S-quintic}
that a singular quintic was exhibited.
We mention that the record Picard number for smooth quintics seems to be 41 by \cite{RS}.
Overall, for surfaces of general type, it is an open problem whether a given deformation class
admits any singular members at all, a notable exception being Horikawa surfaces 
with a congruence condition on the Euler characteristic due to Persson \cite{Persson}.

For surfaces of Kodaira dimension less than two, the situation is quite different.
For instance, on elliptic surfaces over $\PP^1$ with section, 
we can always arrange for enough reducible fibers so that
over some number field the bound \eqref{eq:Lef} is attained.
Similarly, purely inseparable base change may lead to supersingular surfaces
(if we start from a rational elliptic surface, then even to unirational ones).
Below we sketch the situation for K3 surfaces before elaborating on possible methods to 
improve the general upper bounds from \eqref{eq:Igusa} and  \eqref{eq:Lef}.

\subsection{Moduli and Picard numbers of K3 surfaces}
\label{ss:K3}

For K3 surfaces (complex, say, although this restriction is not strictly necessary),
Problem \ref{prob} has a systematic answer in terms of moduli theory.
Namely K3 surfaces with $\rho\geq r$ for given $r\leq h^{1,1}=20$ come in $20-r$ dimensional families.
Explicitly these can be given by a lattice polarisation,
i.e.~each member admits a primitive embedding of the generic N\'eron-Severi lattice 
into the special one (containing the ample class).
Thus we obtain a full solution to Problem \ref{prob} for K3 surfaces over $\C$,
but this result does not a priori extend to any number field.
In fact, it would have been feasible, though unexpected that 
while a very general K3 surface has Picard number one,
all the K3 surfaces over $\bar\Q$ lie on the countably many hyperplanes in the moduli space
comprising K3 surfaces with lattice polarisations of Picard number at least two.
This unlikely behaviour was ruled out for K3 surfaces of any degree 
by Terasoma and Ellenberg:

\begin{Theorem}[Terasoma \cite{Terasoma}, Ellenberg \cite{Ell}]
For any $d\in\N$, there is a K3 surface of degree $2d$ over $\bar\Q$ with Picard number one.
\end{Theorem}

Terasoma's result holds true more generally over $\Q$
and for any complete intersection surface
of non-negative Kodaira dimension,
in particular for quintics.
His work was a natural continuation of Deligne's proof of Noether's conjecture
that a generic hypersurface in $\PP^3$ of degree at least 4 
has Picard number one \cite{SGA7II}.

However, it is a completely different task to exhibit an explicit surface with Picard number one over $\Q$ (or $\bar\Q$).
For quintics and surfaces of higher degree, this was first achieved by Shioda \cite{Sh-PicV}, but for
K3 surfaces only in 2007 by van Luijk  \cite{vL}.
Our arithmetic deformation technique, however, will enable us to engineer K3 surfaces over $\Q$
with $\rho=1$ in abundance as we shall exploit in \ref{ss:K3-1}.
Meanwhile we end this section by reviewing possible methods to improve the general upper bounds
\eqref{eq:Igusa}, \eqref{eq:Lef} for the Picard number.

\subsection{Upper bounds}
\label{ss:upper}

In essence, there are two ways to improve the general upper bounds
\eqref{eq:Igusa}, \eqref{eq:Lef} for the Picard number.
The first relies on the use of non-symplectic automorphisms on a given surface $S$ (over $\C$). 
By analysing their induced action on the holomorphic two-forms of $S$,
one can often endow the transcendental part $T(S)$ of the  Hodge structure on $H^2(S,\Z)$
(i.e.~the minimal sub-Hodge structure whose complexification contains $H^{2,0}(S)$)
with the structure of a module over some cyclotomic integers, or even split it up further.
This often leads to congruence conditions for the rank of $T(S)$ and thus to
rank estimates exceeding the minimum of $2h^{2,0}(S)$
which in turn decreases the upper bound for $\rho(S)$ as pioneered in \cite{Sh-PicV}.
We will briefly come back to this theme in \ref{ss:9},
and also in the special situation of a wild automorphism in positive characteristic in Section \ref{s:6-8}.

Abstractly, one can also do with related, but less geometric information
which is encoded in the Hodge group $E=\mbox{End}(T(S))$,
i.e.~the endomorphism algebra of $T(S)$ respecting the Hodge decomposition.
Upon endowing $T(S)$ with the structure of a vector space over $E$,
one can derive a congruence condition for the rank of $T(S)$ generalising the above argument starting from surface automorphisms.
However,
there is no algorithm known that would determine $\mbox{End(T(S))}$.
For a K3 surface $S$,
Charles worked out an algorithm conditional on the Hodge conjecture for codimension two cycles 
on  $S\times S$ in \cite[Thm.~5 and its proof]{Charles},
but this seems unfeasible to execute in practice.

\subsection{Specialisation}
\label{ss:spec}

The second standard approach towards improving the general upper bounds
 for the Picard number consists in specialisation.
 Here we are mostly concerned with an algebraic surface $S$ over some number field $K$
 and consider a prime $\p$ of good reduction.
 Denoting $S_\p = S\otimes\F_\p$, we obtain a specialisation embedding
 \begin{eqnarray}
 \label{eq:spec}
 \NS(S) \hookrightarrow \NS(S_\p)
 \end{eqnarray}
 which respects the intersection pairing. Directly we infer the upper bound for the Picard number of $S$
 \begin{eqnarray}
 \label{eq:rho-p}
 \rho(S) \leq \rho(S_\p).
 \end{eqnarray}
 The reader may wonder how this really improves our situation in practice.
 The reason lies in the Tate conjecture \cite{Tate}
which states that the algebraic part of $H^2_\text{\'et}(S_\p,\Q_\ell(1))$,
i.e.~the image of the cycle class map \eqref{eq:et},
is exactly the subspace where the absolute Galois group acts through a finite group.
More precisely, if $S$ is defined over some finite field $\F_q$,
then the Tate conjecture postulates that the subspace of $\NS(S_\p)\otimes\Q_\ell$
generated by divisor classes defined over $\F_q$,
is isomorphic to the fixed locus of Galois   in the second cohomology:
\[
\NS(S_\p/\F_q)\otimes\Q_\ell \cong H^2_\text{\'et}(S_\p,\Q_\ell(1))^{\Gal(\bar\F_q/\F_q)}
\]
The Tate conjecture is open in general, but it has recently been proved for K3 surfaces
outside characteristic $2$
by work of Maulik \cite{Maulik}, Charles \cite{Charles-Tate} and Madapusi Pera \cite{Madapusi}
(based on the now classical case of elliptic K3 surfaces from \cite{ASD}).
Independently of the validity of the Tate conjecture,
the embedding \eqref{eq:et} gives an upper bound for $\rho(S)$ as follows:
it is known that $\NS(S_\p)$ is always generated by divisor classes 
defined over a finite extension of the ground field;
this implies that all eigenvalues of Frobenius on the algebraic part inside $H^2_\text{\'et}(S_\p,\Q_\ell(1))$
are roots of unity.
The Tate conjecture postulates that any eigenspace with eigenvalue of Frobenius a root of unity $\zeta$ is algebraic,
but at any rate their number (with multiplicities) gives an upper bound for the Picard number:
if $\alpha_1,\hdots,\alpha_{22}$ denote the eigenvalues of Frobenius on $H^2_\text{\'et}(S_\p,\Q_\ell(1))$,
then
\begin{eqnarray}
\label{eq:number}
\rho(S) \leq \rho(S_\p) \leq \# \{i \in\{1,\hdots,22\}; \exists\, n\in\N: \alpha_i^n=1\}.
\end{eqnarray}
This is of use for practical reasons
because the characteristic polynomial $\chi_\p$ of Frob$_\p^*$ on $H^2_\text{\'et}(S_\p,\Q_\ell(1))$ can,
at least in principle, be computed from point counts over sufficiently many extensions of $\F_\p$
by applying Lefschetz' fixed point formula
(for recent improvement using $p$-adic cohomology, see  \cite{AKR}).
By this means, one can often derive a much better bound than \eqref{eq:Igusa}
resp.  \eqref{eq:Lef}.
However, there are two drawbacks on top of computational matters:
this approach is far from being constructive
and, maybe more importantly, there is a parity condition
imposed by the Weil conjectures
which imply that $\chi_\p$ has coefficients over $\Q$.
In consequence, the right-hand-side number in \eqref{eq:number} always assumes the same parity as $b_2(S)$.
Thus the above method cannot be used directly to prove odd Picard number for some K3 surface,
or even Picard number for some quintic.
We will overcome this for quintics by employing our  constructive approach
of arithmetic deformations, without using any point counting at all,
which will be introduced in Section \ref{s:af}.
Meanwhile we end this section by reviewing what's been used to manoeuvre around 
the parity obstruction for K3 surfaces.

\subsection{Improvements against parity}

It was only in 2007 that van Luijk published an idea
how to prove Picard numbers of the 'wrong' parity based on the above specialisation properties.
Roughly, this requires two different primes
where \eqref{eq:number} gives the same upper bound for $\rho(S)$.
Then one checks whether the N\'eron-Severi lattices of the two specialisations are compatible.
If, for instance, the square classes of their discriminants do not agree, then we infer that the specialisation embedding
\eqref{eq:spec} cannot be of finite index, so $\rho(S)<\rho(S_\p)$.
Van Luijk used this to exhibit a quartic K3 surface of Picard number one \cite{vL}.
Applied to sufficiently many primes,
this approach
would usually succeed in returning the Picard number of a K3 surface $S$ over a number field.
A notable obstacle comes from the endomorphism algebra $\mbox{End}(T(S))$,
see \ref{ss:upper} and for details \cite{Charles}.
We point out that Charles' ideas in \cite{Charles} heavily depend on the fact that $h^{2,0}=1$
for a K3 surface; hence they do not carry over to quintics or most other surfaces of general type.

Subsequent to van Luijk's work, Elsenhans and Jahnel \cite{EJ} pointed out that one can do with specialisation at a single prime
if one additionally studies the obstructions to lifting divisors back from $S_\p$ to $S$.
The proposed method is based on work  of Raynaud on Picard groups \cite{Ray}.
Below we give a simplified statement over $\Z$. 
The general statement for discrete valuations rings in mixed characteristic involves a condition on the ramification degree \cite[Thm.~3.6]{EJ}.

\begin{Theorem}
\label{thm:prim}
Let $S$ be a projective surface defined over $\Q$
and $p>2$ a prime of good reduction. 
Then the specialisation embedding \eqref{eq:spec} is primitive.
\end{Theorem}

In other words, the cokernel of \eqref{eq:spec} is torsion-free.
As a consequence, we find:

\begin{Corollary}
\label{cor:lift}
In the above set-up,
assume that some divisor class does not lift from $S_p$ to $S$.
Then 
\[
\rho(S)<\rho(S_p).
\]
\end{Corollary}
Combined with the upper bound for $\rho(S_p)$ from \eqref{eq:number},
this was used in \cite{EJ} to prove Picard number one for a K3 surfaces 
using reduction and point counting at one prime only.

Our arithmetic deformation technique will use Theorem \ref{thm:prim} as a key ingredient.
In favour of a constructive nature, however, we will try to do without point counting at all.
To this end, we study surfaces where the Picard number can be controlled completely
(without this problem being trivial, as for surfaces with $p_g=0$).
We will set the scene in the next section.


\section{Fermat and Delsarte surfaces}
\label{s:del}

\subsection{Fermat surfaces}
\label{ss:Fermat}

We start with a very special sort of surfaces
whose geometry and arithmetic has been understood quite well for a long time:
Fermat surfaces, defined for given degree $m$ by
 \[
S_m = \{x^{m}+y^{m}+z^{m}+w^{m}=0\}\subset\PP^3.
\]
It goes back to A.~Weil that their zeta functions over finite fields can be described in terms of Gr\"ossencharakters
\cite{Weil}.
Underlying all considerations, there is the automorphism group of $S_m$
which not only contains the symmetric group on four letters,
but also three independent scalings of coordinates by the group $\mu_m$ of $m$-th roots of unity.
Before analysing their action on $S$ and its cohomology, we emphasise 
how the automorphism group led to the discovery of Fermat surfaces 
as arguably the first non-trivial examples of supersingular surfaces:

\begin{Example}
[Tate-Thompson {\cite[p.~102]{Tate}}]
\label{ex:TT}
Let $m'=p^\nu+1$. Then $S_{m'}$ is supersingular over $\bar\F_p$,
since its automorphism group contains the unitary group $U(4)$ over $\F_{p^{2\nu}}$
which acts on $H^2_\text{\'et}(S_{m'},\Q_\ell)$ through two irreducible representations,
one of dimension $1$ and one of dimension $b_2(S_{m'})-1$.
Through the dominant morphism $S_{m'} \to S_{m}$ for $m\mid m'$,
this also  proves supersingularity 
for all Fermat surfaces of degree dividing $m'$,
that is, if $p^\nu\equiv -1\mod m$. 
\end{Example}

Following ideas of Weil, Katz, Ogus and Shioda,
one can decompose the primitive cohomology $\Hp$,
i.e. the orthogonal complement of the hyperplane section $H$
understood in singular or $\ell$-adic cohomology,
into one-dimensional eigenspaces $V(\alpha)$ with character $\alpha$ for the induced action of $\mu_m^3$:
\[
\Hp \cong \bigoplus_{\alpha\in\A_m} V(\alpha).
\]
Here $\alpha$ runs through the following subset of the character group of $\mu_m^3$:
\[
\A_{m}:=
\left\{\alpha=(a_1,a_2,a_3)\in(\Z/m\Z)^3\,|\, a_i\neq 0,\,\sum_{i=1}^3 a_i\neq 0\,\right\}.
\]
For us, it is crucial to decide which eigenspaces $V(\alpha)$ are algebraic
(i.e.~contained in the image of $\NS(S_m)$ under the cycle class map \eqref{eq:et}).
Over $\C$, this can be worked out using the Galois group $\Gal(\Q(\mu_m)/\Q)\cong (\Z/m\Z)^\times$
which acts on the eigenspaces by coordinatewise multiplication:
\[
(\Z/m\Z)^\times\ni t:\;\;V(\alpha)\mapsto V(t\cdot\alpha).
\]
We will need the Hodge type of the eigenspace $V(\alpha)$.
For this purpose, note that $\alpha\in\A_m$ can be given by canonical representatives 
$0<b_1, b_2, b_3<m$ and define 
\[
|\alpha | = \left\lfloor \sum_{i=1}^3 b_i/m\right\rfloor \in\{0,1,2\}.
\]
It can be derived from the induced action of $\mu_m^3$ on $H^{2,0}(S_m)$
that the eigenspace $V(\alpha)$ has Hodge type 
$(2-|\alpha |, |\alpha |)$.
Hence $V(\alpha)$ is algebraic if and only if the full orbit under $\Gal(\Q(\mu_m)/\Q)$ 
has Hodge type $(1,1)$.
That is, $|t\cdot\alpha |=1$ for each member $t\cdot \alpha$ in the Galois orbit of $\alpha$,
written out as $m<\sum_i b_i<2m$ for each member in the orbit.
In particular, we can compute  the Picard number  combinatorially by 
singling out all $\alpha\in\A_m$ whose Galois orbit does not leave the Hodge type $(1,1)$.
Computations become especially transparent in special cases, for instance for degree $m$
relatively prime to $6$ where all of $\NS(S_m)$ is generated over $\Z$ by lines by \cite{Deg}
(see also  \cite{Sh-Fermat} and \cite{SSvL}).

Alternatively, and slightly more direct, one may compute the \emph{Lefschetz number} 
$\lambda(S)$
by summing up over the $\Gal(\Q(\mu_m)/\Q)$-orbits
of all $\alpha\in\A_m$ of Hodge type $(2,0)$,
i.e. with $\sum_i b_i<m$, 
and use
\begin{eqnarray}
\label{eq:b_2}
b_2(S) = \lambda(S) + \rho(S).
\end{eqnarray}
In positive characteristic $p\nmid m$, there is but a small addition necessary
depending on the order $o$ of $p$ in $(\Z/m\Z)^\times$.
Namely, an eigenspace $V(\alpha)$ is algebraic over $\bar\F_p$ if and only if
each member of its Galois orbit is algebraic on average upon multiplying with $p$-powers:
\begin{eqnarray}
\label{eq:o}
\sum_{j=0}^{o-1} |p^j\cdot\alpha | = o \;\;\; \forall\, t\in(\Z/m\Z)^\times.
\end{eqnarray}
Again, we can thus compute Lefschetz number and  Picard number
in a purely combinatorial way. 
We note two special cases:
on the one hand, we recover the supersingularity from Example \ref{ex:TT}.
On the other hand, we find that the Picard number stays constant upon reduction for certain primes (which was also noticed in \cite{Tate}):

\begin{Lemma}[Tate]
\label{lem:rho-Fermat}
Let $S_m$ denote the Fermat surface of degree $m$ and consider a prime $p\equiv 1\mod m$.
Then
\[
\rho(S\otimes\bar\F_p) = \rho(S\otimes\C).
\]
\end{Lemma}
For other primes, one has to use \eqref{eq:o} to work out the Picard number over $\bar\F_p$.

\begin{Example}
\label{ex:55}
On the Fermat surface of degree $m=55$, 
consider the  eigenspace $V(\alpha)$ for $\alpha=(9,11,10)$ of Hodge type $(2,0)$ over $\C$.
Its orbit under $(\Z/m\Z)^\times$ has size $40$.
By \eqref{eq:o}, the orbit is algebraic over $\bar\F_p$ if and only if the residue class of $p$ modulo $m$ 
lies in the following set:
\[
\mathcal H:=\{3, 19, 24, 27, 29, 37, 38, 39, 42, 47, 48, 53, 54\}\subset(\Z/m\Z)^\times.
\]
$\mathcal H$ comprises those primes $p$ with some $\nu\in\N$ such that $p^\nu\equiv -1\mod m$ (cf.~Lemma \ref{lem:del}), 
and all odd powers of $3$ modulo $m$.
\end{Example}


\subsection{Delsarte surfaces}
\label{ss:del}

A Delsarte surface is (the minimal resolution of) a surface in $\PP^3$ 
which can be defined by an irreducible polynomial consisting of four monomials. 
For our purposes,
the crucial property of Delsarte surfaces is
that each is birational to the Galois quotient $S_m/G$ of a Fermat surface $S_m$ 
by a finite subgroup $G\subset\mu_m^3$. 
Here $m$ and $G$ can be determined in terms of the cofactor matrix associated to the exponent matrix of
the defining polynomial following \cite{Sh}.

\begin{Remark}
\label{rem:scale}
For later reference, we point out that the precise coefficients of the four monomials may affect the arithmetic
of the Delsarte surface,
but not its geometry,
as we can always rescale coordinates to normalise the coefficients
over an algebraically closed field.
In fact, this property paves the way to the above use of the discrete data of the exponent matrix of
the defining polynomial.
\end{Remark}

Since the Delsarte surface $S$ arises from $S_m/G$ by resolution of singularities,
if necessary followed by blowing down  $(-1)$-curves until we reach a minimal surface,
it is a non-trivial task to describe the algebraic part of $H^2(S)$ in terms of $S_m$.
Instead, consider the non-algebraic part 
perpendicular to the image of $\NS(S)$ under the cycle class map;
often this is called the transcendental lattice and denoted by
\[
T(S) = \NS(S)^\vee\subset H^2(S).
\]
By \cite{Sh},
this
is governed completely by $S$.
In detail, $T(S)$ is given by the non-algebraic eigenspaces $V(\alpha)\subset H^2(S_m)$
which are invariant under the induced $G$-action.
Explicitly, if $g\in G$ acts on $S_m$ as
\[
 [x,y,z,w] \mapsto [x,\zeta_1y, \zeta_2z,\zeta_3w], \;\;\;\;\;\; \zeta_i\in\mu_m\; (i=1,2,3),
\]
then $V(\alpha)$ is invariant under $g^*$ if and only if
\[
\prod_{i=1}^3 \zeta_i^{a_i} =1\;\;\;\; (\alpha=(a_1,a_2,a_3)).
\]
Again this shows that we can compute the Lefschetz number of $S$ purely combinatorially;
in suggestive notation
\begin{eqnarray}
\label{eq:G}
\lambda(S) = \lambda(S_m/G) =  \lambda(S_m)^G,
\end{eqnarray}
and the Picard number of $S$ follows from \eqref{eq:b_2}.
Using the complete description of the Picard numbers of Fermat surfaces in positive characteristic,
we find the following two basic cases:

\begin{Lemma}
\label{lem:del}
Let the Delsarte surface $S$ be dominated by the Fermat surface of degree $m$.
Consider a prime $p\in\N$.
Then
\begin{itemize}
\item
$\rho(S\otimes\bar\F_p)=\rho(S\otimes\C)$ if $p\equiv 1 \mod m$.
\item
$\rho(S\otimes\bar\F_p)=b_2(S)$ if there is some $\nu\in\N$ such that $p^\nu\equiv -1 \mod m$.
\end{itemize}
\end{Lemma}

As before, for primes outside the above two cases, we have to go through the explicit computations to compare
Picard numbers in characteristic zero and $p$.
For an example, see the K3 surface in \ref{ss:K3-1} or the quintics in Section \ref{s:9-12}
which build on Example \ref{ex:55}.

\subsection{Picard numbers of complex Delsarte surfaces}

In \cite{S-quintic} we went through all quintic Delsarte surfaces
employing the approach of sections \ref{ss:Fermat} and \ref{ss:del}.
We singled out those surfaces with only isolated rational double points as singularities over $\C$
by requiring that the dominating Fermat surface $S_m$
contain four distinct $G$-invariant eigenspaces $V(\alpha)$ of Hodge type $(2,0)$.
Then we computed the Picard numbers of the minimal resolutions with the help of a machine
using \eqref{eq:b_2} and \eqref{eq:G}.
The results are summarized below:

\begin{Theorem}[Sch\"utt {\cite{S-quintic}}]
\label{thm:del}
Complex quintic Delsarte surfaces with only isolated rational double point singularities
attain exactly all the odd numbers between $1$ and $45$ as Picard numbers
except for $3,7,9,11,15$.
\end{Theorem}

This covers a good portion of the Picard numbers
claimed in Theorem \ref{thm},
but no even ones at all.
Yet these Delsarte surfaces will form the crucial starting points for our considerations
thanks to Lemma \ref{lem:del}
which gives us complete control over the Picard numbers of the reductions.
In the next section, this will be combined with the technique of arithmetic deformations
in order to engineer quintics with the  Picard numbers
missing to prove Theorem \ref{thm}.
Meanwhile we end this section by commenting on recent progress
on the Picard numbers of complex Delsarte surfaces of arbitrary degree.

Fermat surfaces admit a closed formula for the Picard number,
depending only on the degree (including some congruence conditions, see \cite{Sh-Fermat}).
It is a non-trivial result that the same can be achieved for Delsarte surfaces.
Here the first subtlety to overcome is the problem
whether the singularities are worse than isolated rational double points.
For quintics, this was solved by computing the geometric genus of the Delsarte surfaces
through invariant eigenspaces of Hodge type $(2,0)$ in \cite{S-quintic},
but recently Heijne found a geometric way around this for general $d$ \cite{Heijne}.
Using this, he reduced the analysis to 83 Delsarte surfaces (up to isomorphism).
For these, Heijne developed a computer code to produce a closed formula for the Picard number,
based on his thesis work on elliptic Delsarte surfaces.

\section{Arithmetic deformations}
\label{s:af}

This section introduces the main technique of this paper: arithmetic deformations.
We will illustrate the method by producing a plentitude of K3 surfaces of Picard number one.
Then we will set the scene for applying the technique to quintics in $\PP^3$
in order to prove Theorem \ref{thm}.

\subsection{Set-up}

The input of the proposed method consists of a surface $S$ over some number field $K$ (ideally $\Q$)
such that we can control the reductions modulo prime ideals $\p$.
Our prototype surfaces will be Delsarte surfaces.
The analysis from \ref{ss:del} allows us to determine those primes $\p$
where
\begin{eqnarray}
\label{eq:=}
\rho(S) = \rho(S_\p)
\end{eqnarray}
(compare Lemma \ref{lem:del}).
Now we deform $S$ into another surface $\mS$ such that the reduction at $\p$ is smooth
and isomorphic to $S_\p$.
This construction comes with an upper bound for the Picard number of $\mS$ by \eqref{eq:spec} and \eqref{eq:=}
\begin{eqnarray}
\label{eq:<=}
\rho(\mS) \leq \rho(\mS_\p) = \rho(S_\p) = \rho(S).
\end{eqnarray}

\textbf{Convention:}
For ease of notation, we will often use the same letter for a single deformation and 
a whole family of deformations.
In what follows, the letter $\mS$ will be reserved for quintics (deforming $S$), $\mX$ for K3 surfaces
(deforming $X$)
and $\mY$ for other surfaces (deforming $Y$).

\subsection{Arithmetic deformation}

We emphasise that if $S$ arises from some quintic in $\PP^3$ by resolving isolated rational double point singularities,
then the deformation $\mS$ will essentially require to preserve the singularities
if we work with the quintic polynomial defining $S$.
This forms a serious obstruction to the possible deformations.
(For two exceptions, see \ref{ss:12}, \ref{ss:def-1}.)

We now come to the constructive nature of the arithmetic deformations.
Namely, if we want to engineer a surface with Picard number $r$,
then we are bound to deform $S$ in such a way
that $r$ independent divisor classes deform to $\mS$.
(If there are singularities involved,
then the exceptional divisors thus can be viewed as constructive and obstructing at the same time.)
By construction, we thus get a lower bound for the Picard number of $\mS$:
\begin{eqnarray}
\label{eq:=>}
\rho(\mS)\geq r.
\end{eqnarray}
In practice, it remains to bridge the gap between lower and upper bound.
We start with the most instructive case where $r=\rho(S)-1$.
Then it is often possible to prove $\rho(\mS)=r$ with the help of Corollary \ref{cor:lift}
(if it applies, for instance for $K=\Q$ and $\p\neq 2\Z$).
To this end,
all we have to do is deform $S$ to $\mS$ in such a way that 
some given divisor class does definitely \emph{not} lift.
In practice, given an explicit curve on $\mS_\p$,
this often amounts to checking that the curve does not lift to second or third order deformations;
for quintics, we will develop some conceptual lifting criteria along these lines in the next section.

Below we give  a brief summary of the arithmetic deformation technique for decreasing the Picard number by one:

\begin{Technique}[Picard number drop by one]
\label{tech-1}
Given an algebraic surface $S$ over some number field
and a prime $\p$ such that $\rho(S)=\rho(S_\p)$,
proceed as follows:
\begin{enumerate}
\item
Deform $S$ $\p$-adically to some surface $\mS$ such that $\mS_\p\cong S_\p$.
\item
Preserve enough divisor classes such that $\rho(\mS)\geq \rho(S)-1$.
\item
Control the deformation of a single divisor class to infer $\rho(\mS)=\rho(S)-1$.
\end{enumerate}
\end{Technique}
We will comment on how to achieve this for quintics in Section \ref{s:lift}.
Below we demonstrate the arithmetic deformation technique
by exhibiting K3 surfaces of Picard number one in abundance.
In \ref{ss:Gal}, we will also explain how it is sometimes possible
to engineer arithmetic deformations of Picard numbers 
provably undermatching the original $\rho(S)$ by more than one
(Technique \ref{tech-2}, see also Technique \ref{tech-3}).

\subsection{K3 surfaces with Picard number one}
\label{ss:K3-1}

To illustrate the arithmetic deformation technique, we 
shall take a little detour from our investigations of quintics
and construct K3 surfaces with Picard number one in abundance.
These surfaces which formed the starting point of all our investigations
were constructed together with Ronald van Luijk in 2010.

Consider the K3 surface $X$
which is given as an affine double sextic
\[
X:\;\;\; w^2 = x^5+xy^5+1.
\]
One easily finds that $X$ admits a rational Galois cover by the Fermat surface $S_{25}$. 
Hence one can compute the Lefschetz and Picard number of $X$ in terms of invariants of $S_{25}$.
Over $\C$ one finds
\[
\rho(X\otimes\C)=2.
\]
Here $\Pic(X\otimes\C)$ is generated by the two components of the pre-image of the 
tritangent line $L=\{x=0\}\subset\PP^2$.
In fact, the same Picard number holds for the reduction $X_p$ 
at any prime $p\equiv 1\mod 5$ using the arguments from \ref{ss:Fermat}, \ref{ss:del}
(compare Lemma \ref{lem:del}). 
Now we  deform $X$ to some double sextic $\mX$ over $\Q$
such that the reduction modulo $p$ is  $X$:
\[
\mX_p = X_p.
\]
Then $\rho(\mX)\leq 2$ by the specialisation embedding \eqref{eq:spec},
and by Theorem \ref{thm:prim} equality can only hold if \emph{all} divisors from $X_p$ lift to $\mX$.
Presently this concerns the two rational curves above $L$
where one can work out the obstructions to lifting explicitly (cf.~\cite{EJ}).
Here one first uses Riemann-Roch to prove that the divisor class of each curve, if it were to lift, 
lifts to an effective divisor which thus again lies over a tritangent.
For instance one finds using second order deformations  that the rational curves do not lift to
the arithmetic deformation
\[
\mX: \;\;\; w^2 = x^5+xy^5+py^4+1.
\]
Including further deformation terms, we find the following families of K3 surfaces of Picard number one:
\begin{Theorem}[Sch\"utt, van Luijk]
\label{thm:K3}
Let $p\equiv 1\mod 5$ and $f\in\Q[x,y]$ of degree at most $6$
with all coefficients $p$-adically integral and the coefficient of $y^4$ a $p$-adic unit.
Then the K3 surface $\mX$ given affinely by
\[
\mX: \;\;\; w^2 = x^5+xy^5+1+pf(x,y)
\]
has Picard number $\rho(\mX)=1$.
\end{Theorem}

%

\subsection{Galois-equivariance}
\label{ss:Gal}

We conclude this section with a  discussion how one can engineer an arithmetic deformation $\mS$
of a given (Delsarte) surface $S$ such that
\[
\rho(\mS)<\rho(S)-1.
\]
To this end, we take up an idea from \cite{EJ-Gal} based on the fact that the specialisation embedding \eqref{eq:spec}
is Galois-equivariant.
In consequence, if some divisor $D$ lifts from $\mS_\p$ to $\mS$,
then we deduce that each Galois module of $\Pic(\mS_\p)$ over $\F_\p$ to which $D$ contributes
lifts to $\mS$.
Conversely, to prove that some sublattice $M\subset\Pic(\mS_\p)$ does not lift to $\mS$,
it suffices to prove the according statement for each irreducible Galois-submodule of $M$.
Thus we find:

\begin{Technique}
\label{tech-2}
Given an algebraic surface $S$ over some number field
and a prime $\p$ such that $\rho(S)=\rho(S_\p)$,
proceed as follows:
\begin{enumerate}
\item
Deform $S$ $\p$-adically to some surface $\mS$ such that $\mS_\p\cong S_\p$.
\item
Preserve some divisor classes such that 
they generate a Galois-invariant sublattice $L$ of $\Pic(\mS_\p)$ of corank $r\in\N$.
\item
Control the deformation of all irreducible Galois-submodules of 
the quotient $\Pic(\mS_\p)/L$ to infer $\rho(\mS)=\rho(S)-r$.
\end{enumerate}
\end{Technique}

The central point about the reduction to the irreducible Galois-submodules
is that it sometimes allows us to break down the lifting problem to investigations
of a few explicit divisor classes (the generators of the irreducible  Galois-modules)
-- each of which thus might be answered by Technique \ref{tech-1}.
Crucially we will use in this context that we can impose certain Galois actions on the Picard group
of a Delsarte surface
simply by scaling the monomials in the equation (cf.~Remark \ref{rem:scale}).
Later we will employ a similar technique using automorphisms (see \ref{ss:34}).

\section{Lifting divisors on quintics}
\label{s:lift}

We want to apply the arithmetic deformations technique to quintics.
Its input consists of a quintic $S$ over some number field $K$
and the smooth reduction $S_\p$ at some prime $\p$.
In view of Theorem \ref{thm:prim},
the basic problem is to decide whether a given divisor $D_\p$ 
on $S_\p$ of lifts to $S$.
In practice, we will always work with an effective divisor $D_\p$.
There are a few subtleties.
First $S$ may involve the minimal resolution of singularities of a quintic $S_0$ in $\PP^3$:
\[
\pi: S \to S_0\subset\PP^3.
\]
Usually, we suppress this distinction,
but in the following paragraphs it will be relevant.
Similarly, we will call different curves a line:
 both
a line  $\ell\subset S_0\subset\PP^3$ in the honest sense
and its strict transform $\tilde\ell\subset S$.
In contrast, for a hyperplane section, we will ambiguously often only refer to $H$,
but take $\pi^* H$ for a hyperplane section on $S_0$
(which will not be ample on $S$).

Except for \ref{ss:12}, \ref{ss:6}
we will always choose our deformations in such a way that the singularities in $\PP^3$ are preserved 
(so that in particular there is no problem with good reduction).
In other words, all exceptional curves will lift automatically.
Therefore we can always make sure that 
the support of $D_\p$ does not contain any exceptional curves,
but we can also add exceptional curves at our leisure.
Note that no non-constant function on $S$ is supported solely on the exceptional curves,
since otherwise we would obtain a non-trivial relation between them in $\Pic(S)$.
It is a consequence of this fact
that subtraction, and incidentally addition of exceptional curves, 
will enable us to decide on the lifting exclusively on the singular quintic model $S_0\subset\PP^3$.
Along these lines we employ the convention
to take the degree of $D_\p$
as the degree of its push-down in $\PP^3$:
\[
\deg(D_\p) = 
D_\p.\pi^* H = \pi_* D_\p.H.
\]

The second subtlety is to distinguish between lifting the divisor $D_\p$
and the divisor class of $D_\p$ in $\Pic(S_\p)$.
In this context, the crucial issue is whether $D_\p$ lifts to an effective divisor.
In general situations, this need not be the case,
but presently we can prove it for effective divisors of small degree without too much difficulty using Riemann-Roch
for a divisor $D$ on a smooth algebraic surface $S$:
\[
\chi(D) = \chi(\OO_S) + (D^2-D.K_S)/2.
\]

\begin{Proposition}
\label{prop:lift}
In the above setup, assume that the effective divisor $ \mathfrak D$ on $S_\p$
has  $\deg( \mathfrak D)\leq 2$.
If the divisor class of $ \mathfrak D$ lifts to $\Pic(S)$,
then $ \mathfrak D$ lifts to a unique (effective) divisor on $S$.
\end{Proposition}

\begin{proof} 
The proof crucially uses the following two properties of smooth specialisation:
\begin{itemize}
\item
the Euler characteristic is constant upon smooth specialisation
(so we can use Riemann-Roch for $ \mathfrak D$);
\item
the dimensions $h^i(\cdot)$ are upper-semicontinuous in mixed characteristic
(in fact, we will only need this property for $h^0(\cdot)$ which is immediate).
\end{itemize}
We will always assume that $\D$ does not contain any exceptional curves in its support
(since these lift anyway by construction);
in one case, however, we will have to add exceptional curves to $\D$ to 
make our arguments work to prove 
Proposition \ref{prop:lift}.

We start by considering $\deg( \mathfrak D)=1$, that is, $ \mathfrak D$ corresponds to a line $ \mathfrak L$ in $\PP^3$.
By adjunction, using $K_S=H$ we find $ \mathfrak D^2=-3$.
Since $\chi(\OO_S)=5$, Riemann-Roch gives 
\[
\chi( \mathfrak D) = h^0( \mathfrak D)-h^1( \mathfrak D)+h^2( \mathfrak D) = 3.
\]
By Serre duality, we have 
\[
h^2( \mathfrak D) = h^0(K- \mathfrak D)  = h^0(H- \mathfrak L)=2
\]
since the linear system $|H- \mathfrak L|$ is generated by the pencil of planes through the line $ \mathfrak L$.
Now we assume that $ \mathfrak D$ lifts to some divisor class $D\in\Pic(S)$.
Using upper semi-continuity, we find
\[
h^0(D)-h^1(D) = \chi(D) - h^2(D) = 3 - h^0(H-D) \geq 3 - h^0(H - \mathfrak D) =1,
\]
so in particular $h^0(D)>0$, i.e. $D$ is effective.
With degree one, $D$ can only correspond to a line $L$ on $S$.
Its reduction on $S_\p$ satisfies
\[
L_\p. \mathfrak L =  \mathfrak D^2=-3.
\]
Since the intersection number is negative,
the line $L$ lifts $ \mathfrak L$ uniquely to $S$.
This proves Proposition \ref{prop:lift} in degree one.

Let $\deg( \mathfrak D)=2$. Then there are three cases:
\begin{enumerate}
\item
\label{it-1}
$ \mathfrak D$ is an irreducible (plane) conic or the sum of two distinct (planar) lines which intersect on $S$;
\item
\label{it-2}
$ \mathfrak D$ is the sum of two skew lines on $S$;
\item
\label{it-3}
$ \mathfrak D$ is twice the same line.
\end{enumerate}
By Theorem \ref{thm:prim}, case \eqref{it-3} reduces to the degree one case.

In case \eqref{it-1} we have $ \mathfrak D^2=-4$ by adjunction.
Riemann-Roch gives $\chi(\mathfrak D)=2$
while $h^2(\mathfrak D) = h^0(H - \mathfrak D)=1$
since $|H-\mathfrak D|$ is generated by the unique plane containing $\mathfrak D$.
As above, we conclude that any lift $D\in\Pic(S)$ of $\mathfrak D$ is effective,
and since $D_\p.\mathfrak D=-4$,
the divisors $D_\p, \D$ have a common component. 
If $D$ is irreducible, then we are done here.
If $D$ is reducible,
we can subtract a common component (a line)
to find that also the other component is common of square $-3$.
In any case, this shows that
 the divisor $D$ uniquely lifts $\mathfrak D$ to $S$.

The proof for case \eqref{it-2}  involves a little subtlety.
We have $\mathfrak D^2=-6$ and $\chi(\mathfrak D)=1$.
If the two lines are not planar in $\PP^3$,
we can argue completely analogously to case \eqref{it-1}
with the adjustment that $h^0(H-\D)=0$. 
If the lines are planar in $\PP^3$, however,
intersecting in a node of $S_0$,
then the above argument fails.
Instead we modify the divisor $\D$ as follows:
consider the chain of exceptional curves $E_1,\hdots,E_r$
in the fiber of $\pi$ above the node
connecting the components met by the  strict transforms of the two lines.
Let
\[
\D' = \D + E_1+\hdots + E_r.
\]
Then $\D'^2=-4$ by definition, so Riemann-Roch gives $\chi(\D')=2$ as in case \eqref{it-1}.
Note that $H-\D'$ is effective
as an easy calculation reveals
for any hyperplane $H\subset S_0$ through a node
that $\pi^* H$ has each exceptional divisor above the node in its support.
Hence we obtain the first inequality of 
\[
1\leq h^0(H-\D')
\leq h^0(H-\D) = 1.
\]
The last equality holds,
since $|H-\D|$ is generated by the unique plane in $\PP^3$ containing both lines
as there is no non-constant function on $S$ supported on the exceptional curves.
In conclusion, $h^2(\D')=1$, and we infer that any lift $D'$ of $\D'$ is effective as before. 
Since $D_\p'.\D'=-4$,
we can proceed to show subsequently
that any irreducible component of $D'$ uniquely lifts some components 
of $\D'$ (since squares stay negative throughout substracting common components).
In particular, $\D$ has also a unique lift on $S$.
\end{proof}

\begin{Remark}
The same arguments go through for irreducible cubic curves,
that is, smooth or singular plane cubics and twisted cubics in $\PP^3$.
For reducible effective divisors of degree $3$, however,
there are cases with few intersections (and thus small square entering in the Riemann-Roch formula)
which require special treatment.
\end{Remark}

\subsection{Application: deformation of two planar lines}
\label{ss:2lines}

Proposition \ref{prop:lift} 
puts us in the position to often decide on the lifting of some effective divisor $\D$ on $S_p$
on the singular model $S_0\subset\PP^3$.
We shall illustrate this by
explicitly deforming two planar lines
to a quadric. 
Throughout we retain the set-up and notation from the previous paragraphs.

\begin{Corollary}
\label{cor:2lines}
Assume that there is some hyperplane $\mathfrak H\subset S_\p$ 
which splits as
\[
\mathfrak H = \mathfrak L_1 + \mathfrak L_2 + C_\p + (\text{exceptional curves}),
\]
where $\mathfrak L_1, \mathfrak L_2$ are lines and $C_\p$ is a cubic without multiple components lifting to $S$.
Denote the lift by $C$ and the unique hyperplane in $\PP^3$ containing $C$ by $H$.
Then either line lifts to $S$ if and only if the quadric residual to $C$ in $H\cap S$ is reducible.
\end{Corollary}

\begin{proof}
Denote the  quadric residual to $C$ in $H\cap S$ by $Q$.
By construction, $Q$ lifts $\mathfrak L_1+\mathfrak L_2$ in $\Pic(S)$ up to some exceptional curves
$E_1,\hdots,E_r$
(if $\mathfrak L_1, \mathfrak L_2$ meet in a node on $S_0$, 
compare case \eqref{it-2} in the proof of Proposition \ref{prop:lift}).
If the divisor class of $\mathfrak L_i$ lifts to $\Pic(S)$,
then $\mathfrak L_i$ lifts to a line $L\subset S$ by Proposition \ref{prop:lift}.
We have
\[
L.Q = \mathfrak L_i.(\mathfrak L_1+\mathfrak L_2+E_1+\hdots+E_r)= -3+1=-2,
\]
so the line $L$ is a component of the quadric $Q$.
\end{proof}

\subsection{Example: Picard number four}

We start with the following Delsarte surface $S$ of Picard number $\rho(S)=5$ over $\C$:
\begin{eqnarray}
\label{eq:5}
S:\;\; x(x^4-y^4) = yz^4+zw^4.
\end{eqnarray}
The Picard number is computed by the techniques from \ref{ss:del}
via the covering Fermat surface $S_m$ which in the present situation has degree $m=64$.
As for generators of $\Pic(S)\otimes\Q$, we can take the five lines
into which the hyperplane $\{z=0\}$ decomposes on $S$
as these generate a lattice of rank 5 and discriminant $256$.
By Lemma \ref{lem:del} certain reductions have the same Picard number:
\begin{eqnarray}
\label{eq:64}
\rho(S_p) = 5 \;\;\; \forall \, p\equiv 1\mod 64.
\end{eqnarray}
Now we employ Technique \ref{tech-1}
to engineer an arithmetic deformation $\mS$ of $S$ which will have Picard number four.
To this end, let
$p\equiv 1\mod 64$ as above
and $f\in\Q[x,y,w,z]$ be a homogeneous quadratic polynomial
whose coefficients are $p$-adic intergers.
Consider the quintic surface
\[
\mS:\;\; x(x^2+y^2)(x^2-y^2+pf(x,y,z,w)) = yz^4+zw^4.
\]
Then $\mS_p=S_p$ and hence $\rho(\mS)\leq 5$ by the specialisation embedding \eqref{eq:spec}
applied to \eqref{eq:64}.
On the other hand, $\rho(\mS)\geq 4$ by the decomposition of the hyperplane $\{z=0\}$
intersected with $\mS$.
By Corollary \ref{cor:2lines},
either line $\{z=x\pm y=0\}$ lifts to $\mS$ if and only if the quadric
$x^2-y^2+pf(x,y,0,w)$ is reducible over $\C$.
Thus, if $f$ is chosen such that the above quadric is irreducible, for instance for $f=w^2$,
then $\rho(\mS)=4$ by Corollary \ref{cor:lift}.

\subsection{Picard number three}
\label{ss:3}

Building on the above example,
we can also construct a quintic of Picard number three
by endowing $S$ with a Galois action on the lines
and then applying Technique \ref{tech-2}.
In detail,
we consider some prime $p\equiv 1\mod 64$ as before
and scale the coefficient of $xy^4$ by some $a\in\Q$ such that
$u^4-a$ is irreducible in $\F_p[u]$.
Now let $f\in\Z[x,y,z,w]$ be homogeneous of degree two and consider the 
arithmetic deformation defined by
\[
\mS:\;\;\; x (\underbrace{x^4-ay^4+2px^2f(x,y,z,w)+p^2f(x,y,z,w)^2}_{g(x,y,z,w)}) = yz^4+zw^4.
\]
Then the quartic
\[
V=\{z=g(x,y,0,w)=0\}\subset\PP^3
\]
on $\mS$ splits into two conjugate conics $Q_1, Q_2$ over the extension $\Q(\sqrt{a})$;
hence $\rho(\mS)\geq 3$.
On the other hand, by Remark \ref{rem:scale}, 
\[
\rho(\mS)\leq \rho(\mS_p) = \rho(S_p)=5.
\]
To prove that really $\rho(\mS)=3$,
we  investigate the Galois-module structure of $\Pic(\mS_p)$.
By construction, there is a single line $\ell=\{x=z=0\}$ defined over $\F_p$ while the four lines 
which the quartic $V_p$ decomposes into, form one orbit with cyclic Galois action.
Thus $\Pic(\mS_p)/\Z\ell_p$ has the following 
 irreducible Galois-submodules
 with some fixed root $\beta\in\F_{p^4}$ of $u^4-a$:
 
\begin{table}[ht!]
$$
\begin{array}{cc}
\text{generators} & \text{rank}\\
\hline
(Q_1)_p + (Q_2)_p & 1\\
(Q_1)_p - (Q_2)_p & 1\\
\begin{matrix}\{z=0,x=\beta y\}-\{z=0,x=-\beta y\},\\
 \{z=0,x=\sqrt{-1}\beta y\}-\{z=0,x=-\sqrt{-1}\beta y\} 
 \end{matrix}
 & 2
\end{array}
$$
\caption{Irreducible Galois-submodules of $\Pic(\mS_p)/\Z\ell_p$}
\label{tab:3}
\end{table}

Since the two rank one Galois modules lift to $\mS$ by definition,
$\mS$ can only have $\rho(\mS)>3$ if the full rank 2 Galois module lifts as well.
That is, all lines lift by Theorem \ref{thm:prim}.
Equivalently, by Corollary \ref{cor:2lines}, both quadrics $Q_1, Q_2$ are reducible over $\C$,
but this we can rule out by the choice of $f$,
for instance by setting $f=w^2$.

\subsection{Picard number two}
\label{ss:2}

We start by adapting the approach of Technique \ref{tech-2} from  \ref{ss:3} to our situation.
Unfortunately there will not be a completely algebraic description of 
quintics with Picard number two,
so we will be quite a bit more specific from the very beginning.

In the notation and set-up from \ref{ss:3}, 
consider the arithmetic deformation
\begin{eqnarray}
\label{eq:2}
\mS:\;\;\; x(x^4-ay^4+px^2w^2) = yz^4+zw^4.
\end{eqnarray}
At $z=0$ we have the line $\ell$ as before and a residual quartic $V$
which can be seen to be irreducible by the Eisenstein criterion.
Thus  
\begin{eqnarray}
\label{eq:2-5}
2\leq \rho(\mS)\leq \rho(\mS_p)=\rho(S_p)=\rho(S)=5,
\end{eqnarray}
where we want to show that the first inequality is in fact an equality.
Inspecting the Galois module structure of $\Pic(\mS_p)$,
we obtain a picture in complete analogy with Table \ref{tab:3}:
the first module corresponds to $V_p$ (a sum over 4 lines on $\mS_p$) and obviously lifts
while the second rank one module corresponds to the alternating sum
over the 4 lines which $V_p$ decomposes into on $\mS_p$.
If $\rho(\mS)>2$, then either the latter Galois module or the rank two module has to lift from $\mS_p$ to $\mS$.
In either case, this implies using the lifting of $V$ and Theorem \ref{thm:prim} 
that some quadric (a sum of two planar lines) lifts to $\mS$.
By Proposition \ref{prop:lift}, the lift is unique and lies in a hyperplane containing $\ell$,
that is, in some 
\[
H_\lambda = \{z=\lambda x\},\;\;\; \lambda\in\PP^1.
\]
This pencil of hyperplanes defines a genus 3 fibration
\[
\pi: \mS \to \PP^1,
\]
which we claim to have no reducible fibers, 
and we have already seen that the quartic $V$ at $\lambda=0$ is irreducible.
To prove this, we will make use of the following analogue of the 
Shioda-Tate formula for elliptic fibrations (cf.~\cite[Cor. 1.5]{Sh-EMS}):

\begin{Lemma}
\label{lem:st}
Denote by $m_\lambda$ the number of components of the fiber of $\pi$ at $\lambda\in\PP^1$.
Then
\[
\rho(\mS) \geq 2 + \sum_{\lambda\in\PP^1} (m_\lambda-1).
\]
\end{Lemma}

\begin{proof}
We have $F^2=0$ for any fiber of $\pi$ (generally part of the moving lemma).
Hence $\ell$ and $F$ generate an indefinite rank two sublattice of $\Pic(\mS)$.
Its orthogonal complement in $\Pic(\mS)\otimes\Q$ is naturally identified with the quotient $F^\perp/\Q F$.
By construction, the fibre components contribute to this quotient.
More precisely, by Zariski's lemma,
quotienting by $F$ kills the only relation between fiber components, 
and the image of the components of any fiber $F_\lambda$ in $F^\perp/\Z F$
is a negative definite lattice of rank $m_\lambda-1$.
Together with $\langle\ell,F\rangle$, 
the local contributions from the fibers sum up to the claimed lower bound for $\rho(\mS)$.
\end{proof}


We continue to rule out other reducible fibers.
To this end, we simply compute (using resultants) those values $\lambda$
where the fiber of $\pi$ attains a singularity.
Outside $\lambda=0,\infty$, they appear as roots of the following polynomial:
%
%
%
%
\[
h(\lambda) = 27 \lambda^{19}+256 a \lambda^3+192 a p^2 \lambda^2+48 a p^4 \lambda+4 a p^6.
\]
By construction, it is exactly these places (including $\infty$) where $\pi$ may have a reducible fibre.
To engineer an arithmetic deformation $\mS$ with Picard number two,
we proceed by
picking $p, a$ as in \ref{ss:3}
in such a way  that $h$ 
is irreducible over $\Q$.
For instance, this holds for $(p, a)=(193,5)$.
This implies that the fibers of $\pi$ at the roots of $h$ do all have the same type;
in particular, they are either all reducible or all irreducible.
But  if they were all to be reducible, then each fiber would contribute to $\Pic(\mS)$
by Lemma \ref{lem:st}:
\[
\rho(\mS)\geq 2+19.
\]
Since this contradicts \eqref{eq:2-5},
we infer that all the fibers of $\pi$ at the roots of $h$ are irreducible,
and the only reducible fiber could sit at $\lambda=\infty$
where the Eisenstein criterion again proves the opposite.

In summary, we conclude that no hyperplane in the pencil $H_\lambda$ contains a quadric on $\mS$,
so none of the quadrics on $\mS_p$ in question can possibly lift to $\mS$.
By Technique \ref{tech-2} we deduce that $\rho(\mS)=2$.

\begin{Remark}
\label{rem:2}
The symmetry  in \eqref{eq:2} allows for an involution
whose quotient is a K3 surface.
On the K3 surface, $\pi$ induces an elliptic fibration,
so one could also analyse the reducible fibers of $\pi$
through the quotient fibration.
\end{Remark}

\subsection{Deforming Singularities}
\label{ss:sing}

At two instances (\ref{ss:12}, \ref{ss:def-1}),
we will allow a singularity to deform smoothly from type $A_{2n}$ to type $A_{2n-1} \; (n\in\N)$.
Here we study the impact on the Picard number
which resembles what we have seen for deforming a pair of lines in \ref{ss:2lines}.

The overall picture is quite simple:
the resolution of the singularity requires $n$ blow-ups of the ambient threespace.
Each of the first $n-1$ blow-ups results in two exceptional curves
on the strict transform of the surface
whose intersection point is a singularity of the surface.
At the final step, the exceptional divisor is a quadric $Q$;
the singularity has type $A_{2n}$ if and only if $Q$ is reducible
and decomposes into two components $E_n, E_{n+1}$.
In particular, this shows that the degeneration from type $A_{2n-1}$ to type $A_{2n}$ is smooth.

We phrase the next lemma in more generality for a surface $Y$
which is either smooth of degree $d\geq 4$ in $\PP^3$
or the minimal resolution of such a surface with isolated rational double point singularities:

\begin{Lemma}
\label{lem:sing}
Let $Y$ be defined over some number field $K$.
Assume that at some prime $\p$ of good reduction,
a singularity of type $A_{2n-1}$ degenerates 
to type $A_{2n}$.
Then
\[
\rho(Y)<\rho(Y_\p).
\]
\end{Lemma}

\begin{proof}
In the  above notation, 
it suffices to show that $E_n$ does not lift from $Y_\p$ to $Y$.
We assume to the contrary that
 $E_n$ lifts to some divisor $\mathfrak E_n$ on $Y$.
We claim that $\mathfrak E_n$ is effective.
To see this we use Riemann-Roch to compute
\[
\chi(\mathfrak E_n) = \chi(E_n)=\chi(\OO_Y)-1 = h^{2,0}(Y) = h^0(K_Y) = \binom{d-1}3.
\]
Here $h^2(E_n)=h^0(K_Y-E_n)=h^0(K_Y)-1$
since the linear system $|K_Y-E_n|$ is given as a subspace of $|K_Y|$ considered over $\PP^3$
by the codimension one condition that the form
vanishes in the singular point.
Hence $h^0(E_n)\geq 1$, 
and upper-semicontinuity applied to $h^2$ implies $h^0(\mathfrak E_n)\geq 1$ as claimed. 

We conclude by observing that the effective divisor $\mathfrak E_n$ has negative intersection with $Q$ on $Y$:
\[
\mathfrak E_n.Q = -1.
\]
Hence $Q$ is reducible, and the singularity type on $Y$ is not $A_{2n-1}$. 
This gives  the required contradiction to our assumption.
In consequence $E_n$ does not lift,
and the claim about the Picard number follows from Corollary \ref{cor:lift}.
\end{proof}

We point out that Lemma \ref{lem:sing}
does not automatically lend itself to simultaneous smooth deformations
of more than one singularity, or a singularity and other divisors (with Picard number dropping by more than one).
It can, however, be applied in conjunction with Techniques \ref{tech-2} and \ref{tech-3}.

\section{Picard numbers 9 through 12}
\label{s:9-12}

Throughout this section, the starting point for our arithmetic deformations
is the following Delsarte surface $S$ of Picard number $\rho(S)=13$,
computed by the method of \ref{ss:del}:
\begin{eqnarray}
\label{eq:S13}
S:\;\;\; x^5+y^5+xzw^3+wz^4 = 0.
\end{eqnarray}
Here $\Pic(S)$ is generated up to finite index by the exceptional divisors above the $A_4$ singularity at $[0,0,0,1]$
and the 10 lines at $z=0$ and $w=0$
defined over the fifth cyclotomic field $\Q(\zeta)$.
The covering Fermat surface $S_m$ has degree $m=55$.
The invariant transcendental cycles on $S_m$ form a single orbit under $(\Z/m\Z)^\times$
which can be represented by $V(\alpha)$ for $\alpha=(9,11,10)$.
With a view towards arithmetic deformations of $S$,
we recall from Example \ref{ex:55} 
that $V(\alpha)$ becomes algebraic
in characteristics $p\neq 2,5,11$
if and only if 
\[
p\in \mathcal H:=\{3, 19, 24, 27, 29, 37, 38, 39, 42, 47, 48, 53, 54\}\subset(\Z/m\Z)^\times.
\]
For all other residue characteristics $p\neq 2,5,11$, we thus find
\begin{eqnarray}
\label{eq:13}
\rho(S_p) = 13 \;\;\; ((p\mod m) \not\in \mathcal H).
\end{eqnarray}
When deforming, we will often aim for preserving the singularity; for instance,
this is guaranteed presently if the deformation term is quadratic in $x,z$.
In contrast, for Picard number 12 we will deform the singularity in order to engineer a quintic over $\Q$
(cf.~Remark \ref{rem:12}).

\subsection{Picard number 12}
\label{ss:12}

Let 
$p\neq 5,11$ with residue class modulo $m$ outside $\mathcal H$.
Consider the arithmetic deformation
\[
\mS: \;\;\; (x+y+pw) (x^4-x^3y+x^2y^2-xy^3+y^4)+ xzw^3+wz^4=0
\]
which preserves the 10 lines and deforms the $A_4$ singularity to type $A_3$.
By inspection of the present algebraic curves, we have $\rho(\mS)\geq 12$.
On the other hand, Lemma \ref{lem:sing} gives
\[
\rho(\mS)< \rho(\mS_\p)=\rho(S_\p)=\rho(S)=13.
\]
Hence, $\rho(\mS)=12$ as desired.


\begin{Remark}
\label{rem:12}
We could have  also engineered a quintic with Picard number 12
by deforming two lines to a quadric while preserving the other lines and the type of the singularity.
However, this can only be achieved over $\Q(\sqrt 5)$.
For instance, letting  $\alpha=\zeta+\zeta^4$, we could take the arithmetic deformation
\[
\mS: \;\;\; (x+y) (x^2+\alpha xy + y^2 + pz^2) (x^2 -(1+\alpha) xy+y^2) + xzw^3+wz^4=0.
\]
\end{Remark}

\subsection{Picard number 11}
\label{ss:11}

Now we restrict to primes $p\equiv 2,3\mod 5$
which are not in $\mathcal H$.
Then the five lines in $S$ above $w=0$ (and above $z=0$)
give one line $\ell=\{w=x+y=0\}$ over $\F_p$ and a cyclic Galois orbit of the remaining four.
Thus we can apply Technique \ref{tech-2}
to the following arithmetic deformation of $S$:
\[
\mS: \;\;\; (x+y) (x^4-x^3y+x^2y^2-xy^3+y^4+pz^2(2x^2-xy+2y^2)+p^2z^4) + xzw^3+wz^4=0.
\]
By definition, we have $\mS_p=S_p$, so $\rho(\mS)\leq 13$.
On the other hand, the quartic on $\mS$ at $w=0$ splits into two conjugate quadrics $Q_1, Q_2$ over $\Q(\sqrt 5)$,
so 
\[
\rho(\mS)\geq 11.
\]
To see that this is in fact an equality, we appeal to the Galois module structure of $\Pic(\mS_p)$.
As in \ref{ss:3}, taking into account the obvious divisors reducing from $\mS$
(exceptional curves, 6 lines, 2 quadrics), 
we are only an irreducible 
Galois module of rank 2 away from lifting the full Picard group of $\mS_p$.
Hence if $\rho(\mS)>11$, then any line would lift from $\mS_p$ by Corollary \ref{cor:lift}.
This is ruled out by Corollary \ref{cor:2lines}
as both $Q_1, Q_2$ are irreducible by the Eisenstein criterion.
Hence $\rho(\mS)=11$.

\subsection{Picard number 10}
\label{ss:10}

We follow up on the previous arithmetic deformation,
but deform 4 lines to an irreducible quartic:
\[
\mS: \;\;\; (x+y) (x^4-x^3y+x^2y^2-xy^3+y^4+pxz^3) + xzw^3+wz^4=0.
\]
Here the quartic $V$ at $w=0$ is irreducible by the Eisenstein criterion again,
so the residual line $\ell$ and the 5 lines at $z=0$ combined with the $A_4$ singularity
yield 
\[
\rho(\mS)\geq 10.
\]
In order to prove equality,
we use $\rho(\mS_p)=13$ as before.
Upon reduction, the quartic $V$ decomposes into a Galois orbit of four lines.
Arguing with Technique \ref{tech-2} as in \ref{ss:2},
we see that $\rho(\mS)>10$ implies
that some quadric composed of two of the four lines lifts from $\mS_p$ to $\mS$.
Automatically, it lies in some hyperplane $H$ containing $\ell$ and reducing to $w=0$.
That is, there is some minimal number field $K$ with a prime ideal $\p\mid p$ and uniformizer $\pi$ of $\p$ such that 
\[
H = \{w=\pi\lambda (x+y)\} \;\;\; \text{for some $\p$-adic integer } \; \lambda\in K. 
\]
We shall now make a comparison of coefficients with the possible lifts of quadrics 
comprised of the four lines on $\mS_p$.
Working in the hyperplane $H\cong\PP^2_{[x,y,z]}$, we obtain
\[
V + \pi^2\lambda^3xz(x+y)^2 + \pi\lambda z^4 = (x^2+(\zeta+\zeta^i)xy+y^2+\pi q_1)
(x^2+(\zeta^j+\zeta^k)xy+y^2+\pi q_1)
\]
with $\p$-adically integral quadrics $q_1, q_2\in K[x,y,z]$ and $\{i,j,k\}=\{2,3,4\}$. 
Modulo $(x^2, xy, y^2)$, this reduces to
\[
pxz^3+ \pi\lambda z^4 \equiv \pi^2 q_1q_2 \mod (x^2,xy,y^2).
\]
From the coefficient of $xz^3$, we deduce that $\pi^2\mid p$, i.e.~$p$ ramifies in $K$.
But this implies that the Galois representation on $H^2_\text{\'et}(\mS,\Q_\ell)$ is ramified at $p$
which in turn means that $\mS$ has bad reduction at $p$, contradiction.
Thus $\rho(\mS)=10$ as desired.

\subsection{Picard number 9}
\label{ss:9}

In order to engineer a quintic with Picard number $9$,
we take a slightly different approach and deform $S$ from \eqref{eq:S13}
in such a way that the oder 5 automorphism
\[
\varphi: (x,y,z,w) \mapsto (x,\zeta y,z,w)
\]
is preserved.
Since this has eigenvalues $\zeta, \zeta, \zeta, \zeta^2$ on $H^{2,0}(S)$,
it endows the transcendental lattice $T(S)\subset H^2(S,\Z)$ with the structure
of a $\Z[\zeta]$-module (cf.~\ref{ss:upper}).
In particular, this implies 
\begin{eqnarray}
\label{eq:T(S)}
4\mid \mbox{rank}(T(S)).
\end{eqnarray}
(Alternatively, we could have scaled the coefficient of $y^5$ by some $a\in\Q$
which is not a fifth power in $\F_p$, in order to use the Galois module structure
employing Technique \ref{tech-2}.)

In detail, consider the arithmetic deformation
\[
\mS:\;\;\; x^5+y^5+px^3z^2+xzw^3+wz^4 = 0
\]
for some prime $p\neq 2,5,11$ whose residue class modulo $m$ is not in $\mathcal H$.
Since the deformation preserves the $A_4$ singularity and the 5 lines at $z=0$
while $\mS_p=S_p$,
we have 
\[
9 \leq \rho(\mS) \leq \rho(\mS_p) = 13.
\]
By \eqref{eq:T(S)} this corresponds to rank$(T(S))=40$ or $44$,
so in particular $\rho(\mS)>9$ implies $\rho(\mS)=13$.
Hence all lines lift from $\mS_p$ to $\mS$.
Consider the 5 lines on $\mS_p$ in the hyperplane $\{w=0\}$.
Since either intersects the other, they lift to a common hyperplane $H$ in $\mS_p$.
By Corollary \ref{cor:2lines}, this hyperplane is unique,
hence it is defined over $\Q$ and invariant under the automorphism $\varphi$.
Moreover, it lifts $\{w=0\}$, so we obtain
\[
H = \{w= p (ax+cz)\} \;\;\; a,b\in\Q\cap\Z_p.
\]
Substituting into $\mS$, we obtain the polynomial
\[
h = x^5+y^5+px^3z^2+p^3xz(ax+cz)^3+p(ax+cz)z^4.
\]
Regarded as a polynomial in $\C[x,z][y]$, $h$ is reducible 
if and only if its constant coefficient is a fifth power  (lifting $x$):
\[
x^5+px^3z^2+p^3xz(ax+cz)^3+p(ax+cz)z^4 \stackrel{!}{=} (x+pbz)^5 \;\;\; (b\in\Q\cap\Z_p).
\]
But then comparing vanishing orders of $p$ at the coefficients of $x^3z^2$,
we find that the deformation summand $px^3z^2$ cannot be compensated for.
Thus the 5 lines on $\mS_p$ in the hyperplane $\{w=0\}$ cannot lift to $\mS$,
and we deduce $\rho(\mS)=9$.


We postpone Picard numbers 6 through 8 for a later treatment,
since they require additional methods
involving K3 surfaces and wild automorphisms, see Section \ref{s:6-8}.

\section{Picard numbers 14 through 16}
\label{s:14-16}

In this section, we will engineer quintics with Picard numbers 14, 15 and 16
starting from the Delsarte quintic
\[
S:\;\;\; w(x^4-y^4) = yz^4+zw^4
\]
of Picard number $\rho(S)=17$.
The above model has 4 rational double points of type $A_3$ at $[1,\alpha,0,0]$ with $\alpha^4=1$
and 6 lines: 5 at $z=0$ plus the line $\{w=y=0\}$.
Together the exceptional curves and the lines generate a sublattice of $\Pic(S)$ of finite index
and discriminant $13^3$.

The Delsarte surface $S$ is dominated by the Fermat surface of degree $m=52$.
Fixing a prime $p\equiv 1\mod 52$,
Lemma \ref{lem:del} gives
\[
\rho(S_p) = 17.
\]
We shall now deform some lines above $z=0$ on $S$ while preserving the 4 singularities.
To this end, we will always use deformation terms which are multiples of $w^2$.

\subsection{Picard number 16}
\label{ss:16}

The arithmetic deformation 
\[
\mS:\;\;\; w(x^2+y^2)(x^2-y^2+pw^2) = yz^4+zw^4
\]
preserves the singularities and 4 of the 6 lines. Hence
\[
16\leq\rho(\mS)\leq\rho(\mS_p)=17
\]
by the above considerations. Since the quadric $x^2-y^2+pw^2$ deforming two of the lines on $\mS_p$ 
is irreducible over $\C$,
Corollary \ref{cor:2lines} implies $\rho(\mS)=16$ by Theorem \ref{thm:prim}.

\subsection{Picard number 15}
\label{ss:15}

We continue by equipping $S$ with a Galois action on the lines (cf.~Remark \ref{rem:scale})
and deforming two pairs of lines to two conjugate conics 
in order to apply Technique \ref{tech-2}.
To this end, let $a\in\Q$ such that $u^4-a$ is irreducible in $\F_p[u]$.
Consider the arithmetic deformation
\[
\mS:\;\;\; w(x^4-ay^4+2px^2w^2+p^2w^4) = yz^4+zw^4
\]
which has $15\leq \rho(\mS)\leq 17$ by construction.
Arguing with the Galois module structure as in \ref{ss:3}, \ref{ss:11},
we prove that $\rho(\mS)=15$ since either quadric in $z=0$
is irreducible over $\C$.

\subsection{Picard number 14}
\label{ss:14}

We proceed as in the previous section, but deform the 4 lines to a single irreducible quartic over $\C$:
\[
\mS:\;\;\; w(x^4-ay^4+pyw^3) = yz^4+zw^4
\]
Thus $14\leq \rho(\mS)\leq 17$, and to show $\rho(\mS)=14$ we have to rule out any quadric 
comprising two of the conjugate lines lifting from $\mS_p$ to $\mS$.
This can be achieved along the lines of \ref{ss:2} (optionally also using the sign involution in $x$, cf.~Remark \ref{rem:2})
or \ref{ss:10}.
In brief, we assume that some hyperplane $\{z=\pi\lambda w\}$ containing the preserved line $\{z=w=0\}$
splits off two quadrics on $\mS$ lifting pairs of lines from $\mS_p$.
Then a comparison of the coefficients of $yw^3$ leads to the contradiction $\pi^2\mid p$ 
as in \ref{ss:10}.

\section{Picard numbers up to 25}

\subsection{Picard number 18}
\label{ss:18}

We start with the Delsarte surface of Picard number 19 from \cite{S-quintic}:
\[
S:\;\;\; x^3yw+xy^4+yz^4+zw^4=0.
\]
The above model has an $A_{16}$ singularity at $[1,0,0,0]$ and 3 lines:
\[
\ell=\{x=z=0\},\;\; \{y=z=0\}, \;\; \{y=w=0\}.
\]
Together with the exceptional curves, they generate a finite index sublattice of $\Pic(S)$ of discriminant $96$.
Since $S$ is covered by the Fermat surface of degree $m=35$,
Lemma \ref{lem:del} gives
\[
\rho(S_p)=19 \;\;\; \forall\;p\equiv 1\mod 35.
\]
We shall now deform $S$ arithmetically 
in a way that preserves the singularity and the two lines other than $\ell$:
\begin{eqnarray}
\label{eq:18}
\mS: x^3yw+xy^4+yz^4+zw^4+py^5=0.
\end{eqnarray}
By construction, we have
\[
18\leq \rho(\mS) \leq\rho(\mS_p)=19.
\]
We will now show that the line $\ell$ does not deform to $\mS$.
By Corollary \ref{cor:lift},
this will suffice to prove that $\rho(\mS)=18$.

Assume to the contrary that $\ell_p$ lifts to $\mS$.
By the uniqueness in Proposition \ref{prop:lift},
there are $p$-adic integers $a,b,c,d\in\Q$
such that the lift is given by
\[
x=p(ay+bw),\;\;\; z=p(cy+dw).
\]
We substitute into \eqref{eq:18} and solve for the resulting degree 5 polynomial in $y,w$ to vanish identically.
The coefficient of $w^5$ gives $d=0$, and subsequently 
we arrive at a contradiction thanks to the term $py^5$ in \eqref{eq:18}.
Thus $\rho(\mS)=18$.

\subsection{Picard number 20}
\label{ss:20}

Our starting point is the Delsarte surface
\[
S:\;\;\; yzw^3+y^4w+xyz^3+x^5=0
\]
It has singularities of type $A_4$ at $[0,0,0,1]$ and $A_{14}$ at $[0,0,1,0]$.
Together with the hyperplane section and the lines 
\[
\ell_1=\{x=y=0\},\;\;\; \ell_2=\{x=w=0\},
\]
the exceptional divisors generate a lattice of rank 20 and discriminant $-34$.
This falls one short of $\rho(S)=21$
as one can check with the covering Fermat surface of degree $m=34$.
For the missing generator of $\Pic(S)$, we consider a degree 6 curve $C$
which lies on $S$ as a non-complete intersection in the hypersurface $\{y^3+zw^2=0\}$.
In $\PP^3$, the curve has a rational parametrisation
\begin{eqnarray*}
\PP^1\;  & \to & \;\;\;\;\; C\\
~[s,t] & \mapsto & [s^5t,-s^2t^4,s^6,t^6]
\end{eqnarray*}
whose inverse rational function $z/x$ is well-defined outside $[0,0,0,1]$.
On the resolution of the $A_4$ surface singularity,
this point of indeterminacy results in a cusp where $C$ meets an exceptional curve tangentially.
Thus $C\subset S$ has arithmetic genus $p_a(C)=0$,
so adjunction gives $C^2=-6$.
By direct computation, one verifies that $C$ and the previously named curves on $S$
generate a finite index sublattice of $\Pic(S)$ of discriminant $17^2$.

We continue with an arithmetic deformation
preserving singularities and all above curves except for $\ell_2$:
\begin{eqnarray}
\label{eq:20}
\mS:\;\;\; yzw^3+y^4w+xyz^3+x^5=py^2(y^3+w^2z)
\end{eqnarray}
for some prime $p\equiv 1\mod 34$.
By construction, we have
\[
20\leq \rho(\mS)\leq\rho(S_p)=21.
\]
In order to prove $\rho(\mS)=20$, it suffices by Corollary \ref{cor:lift}
to show that the line $\ell_2$ does not lift to $\mS$.
Assuming to the contrary that $\ell_2$ lifts,
Proposition \ref{prop:lift} guarantees uniqueness, so the hypothetical lift $L$ is defined over $\Q$.
Clearly $L$ is contained in the pencil of planes containing $\ell_1$, so we obtain
\[
L:\;\; x=p\lambda y,\;\;\; w=p(ay+bz)
\]
where $\lambda, a,b\in\Q$ are $p$-adic integers. We substitute into \eqref{eq:20} and compare coefficients:
$yz^4$ gives $b=0$, then $y^2z^3$ gives $\lambda=0$,
but then $y^4z$ implies $a=0$ leaving the constant term $py^5$, contradiction.
Thus $\ell_2$ does not lift to $\mS$, and we deduce $\rho(\mS)=20$.

\subsection{Picard number 22}
\label{ss:22}

Consider the Delsarte surface
\[
S:\;\;\; x^3yw+y^5+z^4w+zw^4=0
\]
which has Picard number $\rho(S)=23$ as can be computed from the covering Fermat surface
of degree $m=45$.
The above model in $\PP^3$ has a singularity of type $A_{19}$ at $[1,0,0,0]$.
Together with the five lines at $y=0$,
the exceptional curves generate a finite index sublattice of $\Pic(S)$ of discriminant $27$.
We shall now deform two lines to a quadric
while preserving the other 3 lines and the singularity.
To this end, let $p\equiv 1\mod 45$ and define
\[
\mS: \;\;\; x^3yw+y^5+zw(z+w)(z^2-zw+w^2+pxw)=0
\]
By construction, we have
\[
22\leq\rho(\mS)\leq\rho(\mS_p)=23,
\]
but the two lines on $\mS_p$ underneath the irreducible quadric $\{z^2-zw+w^2+pxw=y=0\}$ on $\mS$
do not lift by Corollary \ref{cor:2lines}.
Hence, by Corollary \ref{cor:lift}, $\rho(\mS)=22$.

\subsection{Picard number 24}
\label{ss:24}

We start with the following Delsarte surface, appropriately scaled for our purposes:
\[
S:\;\;\; w^5-xz^4+2xy^3w+x^5=0.
\]
It is covered by the Fermat surface of degree 60;
using \ref{ss:del} we compute $\rho(S)=25$,
and the same for the reduction $S_p$ modulo any prime $p\equiv1\mod 60$.
The above equation accommodates an $A_{19}$ singularity at $[0,1,0,0]$
and 5 lines at $\{w=0\}$.
Together these curves generate a sublattice of $\Pic(S)$ of rank 23.
For the full Picard lattice, we complement these curves
by the following non-complete intersections:
consider the hypersurface 
\[
\{xw=y^2\}
\]
which splits on $S$ into two smooth degree 5 curves of genus $2$,
given in $\PP^3$ by
\[
x^5+y^5\pm x^3z^2=0.
\]
For each curve, adjunction gives $C^2=-3$.
Applying the oder three automorphism $y\mapsto \zeta_3 y$, we obtain two further such pairs.
With the previous curves,
they generate a sublattice of $\Pic(S)$ of finite index and discriminant $900$.

We shall now engineer an arithmetic deformation which preserves the singularity,
the 5 lines and the first pair of genus 2 curves.
These curves generate  a lattice of rank 24.
In detail, let $a,b\in \Q$ be $p$-adic integers and consider 
\[
\mS:\;\;\; w^5-xz^4+2xy^3w+x^5+pxw(xw-y^2)(ax+bw)=0
\]
Then by definition,
\[
24\leq\rho(\mS)\leq\rho(\mS_p)=25.
\]
In order to prove that the first inequality is attained,
it suffices by Corollary \ref{cor:lift} to prove that one of the other genus 2 curves does not lift.
Indeed, any of these genus 2 curves $C$ on $S$ satisfies
\[
\chi(C) = 1,\;\;\; h^2(C)=0.
\]
Hence one can show as in the proof of Proposition \ref{prop:lift}
that any lift is automatically effective and unique.
However, the deformations of these curves are not so easy to control
(notably because they are not complete intersections).
Instead, we decided to pursue an alternative approach
utilising the involution
\[
\imath: \;\; (x,y,z,w) \mapsto (x,y,-z,w).
\]
As in \ref{ss:2}, the quotient surface $\mS/\imath$ resolves to a K3 surface $\mX$
deforming the resolution $X$ of $S/\imath$.
For our purposes, it is crucial that each genus 2 curve is invariant under $\imath$;
they map to rational curves on the quotient surface.
%
%
Hence we can control their deformations from $S$ to $\mS$
by studying rational curves deforming from $X$ to $\mX$.
Here we work with specific models of $\mX$ and $X$
which not only come with natural double sextic models, 
but also with elliptic fibrations expressed in the invariant coordinate $u=z^2$:
\[
\mX:\;\;\; u^2=x(w^5+2xy^3w+x^5+pxw(xw-y^2)(ax+bw)).
\]
In detail, we dehomogenise by setting $x=1$ and convert to Weierstrass form.
For $X$, this results in the isotrivial elliptic fibration
\[
X:\;\;\; u^2 = y^3 + 4w^2(w^5+1)
\]
which makes the automorphism $y\mapsto \zeta_3 y$ visible.
Since $X$ is also a Delsarte surface,
one computes directly that $\rho(X)=14$,
corresponding to an invariant transcendental cycle on $S$ of Hodge type $(2,0)$ 
and orbit length $8$. 
There are reducible fibers of type $IV$ at $w=0$ and $II^*$ at $w=\infty$.
Moreover there is a total number of six sections with $y=w^4$ or $\zeta_3 w^4$
pulling back from a rational elliptic surface via the base change $s=t^5$.
From the theory of Mordell-Weil lattices \cite{ShMW}, it follows
that these sections generate the full Mordell-Weil group
whose lattice structure is 
\[
\MWL(X)=A_2^\vee(5).
\]
This holds because it results in $\NS(X)$ having discriminant $-25$
which as an even lattice of rank $14$ prevents any integral even overlattice.
Incidentally, the above sections are exactly induced from the genus 2 curves on $S$
which give bisections on $X$.
The  elliptic fibration deforms to $\mX$ with the same reducible fibers, but
without preserving isotriviality (since the automorphism is not preserved on $\mS$):
\begin{eqnarray}
\label{eq:a+bw}
\mX: \;\;\; u^2 = y^3 -pw(a+bw)y^2 + 4w^2(w^5+1+pw^2(a+bw)).
\end{eqnarray}

\begin{Lemma}
\label{lem:24}
If $p\nmid b$, then $\rho(\mX)=13$.
\end{Lemma}

\begin{proof}
The genus 2 curves in the hypersurface $\{xw=y^2\}$ on $\mS$ still induce a pair of sections on $\mX$,
this time with $y$-coordinate $w(w^3+p(a+bw))$.
By construction, we have 
\[
13\leq\rho(\mX) \leq \rho(\mX_p) = 14.
\]
If $\rho(\mX)=14$, then also the sections with $y=\zeta_3 w^4$ on $X$ 
would deform to $\mX$ by Corollary \ref{cor:lift}.
But then the theory of elliptic surfaces with section,
as recorded in the Mordell-Weil lattices, predicts that the deformations
take exactly the same shape as the original sections in order to preserve intersection numbers.
This  fact which was largely utilised in the determination of sections 
on some specific K3 surfaces in \cite{ES}, for instance,
will be central for us to achieve the exact opposite, that is, disprove lifting.
Namely, we can assume that any lift of the above section has $y$-coordinate 
a polynomial in $w$ of degree $4$.
Explicitly, fixing some $r\in\F_p$ such that $r^3=1, r\neq 1$, 
there is a minimal number field $K$ and a  prime $\p$ above $p$
such that the section is given by  polynomials
in $K[w]$ of degree 4 resp. 6 with $\p$-adically integral coefficients
 lifting $r w^4$ resp. $w(w^5+2)$.
 Here $p$ cannot ramify in $K$ because $\mX$ has good reduction at $p$ and $K$ is minimal, 
 so we can work with a uniformiser $\xi$ of $\p$
 and compare vanishing order modulo $\xi$ as in \ref{ss:10}.
Substituting into \eqref{eq:a+bw},
we can spell out the equations to see that the sections lift uniquely to order $\xi^2$, 
but not to order $\xi^3$ if $\xi^2\nmid pb$, i.e. $p\nmid b$.
\end{proof}

As a consequence of Lemma \ref{lem:24}, we can arrange for $\mS$ to have  $\rho(\mS)=24$
  by picking $b$ not divisible by $p$.

\section{Higher Picard numbers}
\label{s:high}

For higher Picard number, it becomes harder and harder to set up the deformations.
In particular, in order to decrease the Picard number exactly by one,
we have to preserve singularities and plenty of divisors.
Starting from Delsarte surfaces, this is often enough impossible.
In this section, we work out arithmetic deformations for select Picard numbers:
\[
\rho=26, 28, 32, 34, 36, 38.
\]
This completes the Picard numbers required to prove Theorem \ref{thm}
except for $\rho=6,7,8,30, 40, 42, 44$ which will be the subject of the next two sections.

\subsection{Picard number 28}
\label{ss:28}

We take off with the Delsarte surface
\[
S:\;\;\; zw^4-z^5+y^3zw+x^3yw=0
\]
which is covered by the Fermat surface of degree $36$.
The Picard group of rank $29$ is generated up to finite index by the exceptional curves over the singularities
of type $A_{14}$ at $[0,1,0,0]$, $A_4$ at $[1,0,0,0]$ 
and $A_2$ at $[0,1,0,r]$ with $r^3=1$, 
together with the seven lines $\{z=x=0\}, \{z=w=0\}$ and at $y=0$. 
We deform the surface for some prime $p\equiv 1\mod 36$ by
\[
\mS: \;\;\; zw^4-z^5+y^3zw+x^3yw=0=pxzw(z^2-w^2).
\]
This preserves the singularities and all lines except for the two deforming to the quadric
\[
\{y=z^2+w^2+pxw=0\}.
\]
In fact, since the quadric is irreducible over $\C$, we already obtain $\rho(\mS)=28$ by Corollary \ref{cor:2lines}
in conjunction with Corollary \ref{cor:lift}.

%

\subsection{Picard number 26}
\label{ss:26}

We essentially continue with the above surface and a prime $p\equiv 1\mod 36$,
but scale the equation
by a $p$-adic integer $a\in\Q$ such that $u^4-a$ is irreducible over $\F_p$:
\[
S:\;\;\; zw^4-az^5+y^3zw+x^3yw=0.
\]
This endows the four lines $\ell_1,\hdots,\ell_4$ at $y=w^4-az^4=0$ on the reduction $S_p$ 
with the same Galois module structure as in \ref{ss:3},
so we can employ Technique \ref{tech-2} to deform $S$ to
the following surface of Picard number $26$:
\[
\mS: \;\;\; zw^4-az^5+y^3zw+x^3yw=0=px^2z^2w.
\]
To see that $\rho(\mS)\geq 26$, one checks that the singularities as well as the 3 lines at $z=0$
are preserved.
If $\rho(\mS)>26$, then by the Galois module structure and Proposition \ref{prop:lift},
some conic  $\ell_1+\ell_j\; (j=2,3,4)$ lifts to a unique conic on $\mS$.
Necessarily this sits inside a hyperplane $H_\lambda=\{y=\lambda z\}$ 
containing the line $\{y=z=0\}$.
It remains  to rule out that the residual quartic
\begin{eqnarray}
\label{eq:pencil-26}
w^4-az^4+\lambda^3z^3w+\lambda x^3w = px^2zw
\end{eqnarray}
splits into two quadrics at some $\lambda$.
We spell out the lifting explicitly in some number field $K$ with uniformiser $\pi$ of
 $p$ which we can assume to be unramified. 
For the rank 1 Galois module corresponding to Table \ref{tab:3},
we have quadrics
\[
w^2-\sqrt{a}z^2+\pi q_1(x,z,w), \;\;\; w^2+\sqrt{a}z^2+\pi q_2(x,z,w)
\]
with $\pi$-adically integral coefficients.
Upon multiplying, we compare coefficients with \eqref{eq:pencil-26}.
At $x^2zw$, we obtain $\pi^2\mid p$, i.e.~$p$ is ramified in $K$, contradiction.

Similarly, the rank 2 Galois modules from Table \ref{tab:3}
give quadrics
\[
w^2-(1+i)\sqrt[4]awz+i\sqrt{a}z^2+\pi q_1(x,z,w), \;\;\; w^2+(1+i)\sqrt[4]awz-i\sqrt{a}z^2+\pi q_2(x,z,w).
\]
As before, we expand and compare coefficients with \eqref{eq:pencil-26}.
Modulo $(\pi^2,w^2,z^3)$, we obtain
\begin{eqnarray}
\label{eq:coeff-26}
\lambda wx^3-px^2zw \equiv \pi (1+i) \sqrt[4]a (q_1-q_2)wz - \pi i \sqrt a z^2 (q_1-q_2).
\end{eqnarray}
From the coefficients of $x^2zw$ we derive that 
\[
q_1-q_2=ux^2+\hdots
\]
for some $\pi$-adic unit $u\in K$.
But then there is no term in \eqref{eq:coeff-26}
to compensate for $-\pi i\sqrt a ux^2z^2.$
This establishes the desired contradiction.
Thus $\rho(\mS)=26$ as claimed.

\subsection{Picard number 32}
\label{ss:32}

Consider the Delsarte surface given by
\[
S:\;\;\; w^5+xy^3z-xyz^3+x^3zw=0.
\]
It has five lines in the hyperplane $w=0$ and the following singularities:
\begin{itemize}
\item
$A_{14}$ at $[1,0,0,0]$;
\item
$A_4$ at $[0,1,0,0], [0,0,1,0], [0,1,1,0], [0,1,-1,0]$.
\end{itemize}
Together they generate a finite index sublattice of $\Pic(S)$ of rank $33$ and discriminant $11^2$;
here one uses the covering Fermat surface of degree $22$
which also implies that $\rho(S_p)=33$ for all $p\equiv 1\mod 22$.
We can deform the surface $S$ while preserving all singularities by monomials involving $x^2z^2$.
In particular, we find the following deformation which at the same time preserves three of the five lines on $S$:
\[
\mS:\;\;\; w^5+xy^3z-xyz^3+x^3zw=px^2yz^2.
\]
By construction, we have
\[
32\leq \rho(\mS)\leq\rho(\mS_p)=33.
\]
Since the quadric $y^2-z^2-pxz=0$ residual to the three lines in the hyperplane $w=0$ is irreducible over $\C$,
we deduce $\rho(\mS)=32$ from Corollary \ref{cor:2lines}
and Corollary \ref{cor:lift}.

\subsection{Picard number 36}
\label{ss:36}

Covered by the Fermat surface of degree $20$,
the Delsarte surface to deform is defined by
\[
S:\;\;\; yzw^3+xyz^3-xy^3z+x^4w=0.
\]
It contains 7 lines: 5 given by $xyz(z^2-y^2)=0$ at $w=0$ and two more at $x=0$.
We compute the following singularities:
$$
\begin{array}{c|c|c}
A_{12} & A_3 & A_2\\
\hline
[0,1,0,0], [0,0,1,0] & [0,0,0,1] & [0,1,1,0], [0,1,-1,0]
\end{array}
$$
Together these rational curves generate a sublattice of $\Pic(S)$ of rank $35$
which is two off $\rho(S)$.
Additionally we consider 4 quadrics 
which are components of the intersections of $S$ with the hyperplanes $y=\pm z$,
such as
\[
Q_\pm = 
\{ y+ z = x^2\pm yw
=0\}
\]
In total,
this gives generators of a finite index sublattice of $\Pic(S)$ of discriminant $625$.
We now deform $S$ using some prime $p\equiv 1\mod 20$ such that
$\rho(S_p)=37$ by Lemma \ref{lem:rho-Fermat}.
In detail, consider
\[
\mS:\;\;\; yzw^3+xyz^3-xy^3z+x^4w=px^2yz(y-z).
\]
One checks that all singularities are preserved as well as all 7 lines;
only one's equations are actually deformed to $\ell=\{w=y+z-px=0\}$.
Moreover, the quadrics in the hyperplane $y=z$ are also preserved,
so
\[
36\leq \rho(\mS)\leq\rho(\mS_p)=37.
\]
In order to prove that the first inequality is attained,
it remains to analyse the possible deformations of $Q_\pm$;
 these would be residual to  $\ell$ in some hyperplane
$\{y+z-px=\lambda w\}$ on $\mS$.
Note that both $Q_\pm$,
considered in $\PP^3$, meet the node at $[0,0,0,1]$.
Thus any lift from $\mS_p$ to $\mS$ would have to do so, too.
This implies $\lambda=0$,
but then the residual quartic $x^4+yw^2(px-y)$ is irreducible over $\C$.
Using Proposition \ref{prop:lift}, we thus conclude that the quartics $Q_\pm$ do not lift to $\mS$.
Hence $\rho(\mS)=36$ by Corollary \ref{cor:lift}.

\subsection{Picard number 34}
\label{ss:34}

In order to exhibit a quintic of Picard number 34,
we rescale the above Delsarte surface to the default coefficients and deform as follows:
\[
\mS:\;\;\; yzw^3+xyz^3+xy^3z+x^4w=px^3yz.
\]
The deformation preserves the singularities and 5 of the lines, so $\rho(\mS)\geq 34$.
On the other hand, $\rho(\mS_p)=37$ by Lemma \ref{lem:rho-Fermat},
and since the two remaining lines $\ell_1, \ell_2$ are deformed to the quadric
\[
\{w=y^2+z^2-px^2=0\}
\]
which is irreducible over $\C$, 
we already deduce $\rho(\mS)<37$ by Corollary \ref{cor:2lines} and Corollary \ref{cor:lift}.
To prove that $\rho(\mS)=34$,
we adapt Technique \ref{tech-2} to our situation
by replacing the Galois module  structure of $\Pic(\mS)$ by the module structure 
with respect to a
subgroup of $\Aut(\mS)$:

\begin{Technique}
\label{tech-3}
Given an algebraic surface $S$ over some number field
with a finite subgroup $G\subset\Aut(S)$
and a prime $\p$ such that $\rho(S)=\rho(S_\p)$,
proceed as follows:
\begin{enumerate}
\item
Deform $S$ $\p$-adically to some surface $\mS$ such that $\mS_\p\cong S_\p$
and $G\hookrightarrow\Aut(\mS)$.
\item
Preserve some divisor classes such that 
they generate a $G$-invariant sublattice $L$ of $\Pic(\mS_\p)$ of corank $r\in\N$.
\item
Control the deformation of all irreducible $G$-submodules of 
the quotient $\Pic(\mS_\p)/L$ to infer $\rho(\mS)=\rho(S)-r$.
\end{enumerate}
\end{Technique}

In what follows, we will use the group 
\[
G=\{ \imath_1, \imath_2, \imath_3, \id\}\cong \Z/2\Z \times\Z/2\Z
\]
consisting of involutions leaving the projective coordinates $x, w$ invariant 
and acting on the projective coordinates $y,z$ as
\[
\imath_1(y,z) = (z,y),\;\; \imath_2(y,z) = (-y,-z),\;\; \imath_3(y,z)=(-z,-y).
\]
Clearly the 5 lines and the exceptional curves define a $G$-invariant submodule $L$ of $\Pic(\mS)$ of rank 34.
The quotient $\Pic(\mS_p)/L$ can be decomposed into rank one $G$-modules in terms of $S$.
Let $\zeta^4=-1, i=\zeta^2$ and the quadrics
$Q_{1}, Q_{2}\subset S$ be defined by
\[
Q_{1} = \{y=iz, x^2=\zeta yw=0\}, \;\;\; Q_{2} = \{y=-iz, x^2=\zeta^3 yw=0\}.
\]
Then the irreducible $G$-submodules of $\Pic(S)/L$ have the following generators and
induced $G$-action:
$$
\begin{array}{c|c|c|c}
\text{generator} & \imath_1^* & \imath_2^* & \imath_3^*\\
\hline
\ell_1-\ell_2 & -1 & 1 & -1\\
Q_{1}+Q_{2} & 1 & -1 & -1\\
Q_{1}-Q_{2} & -1 & -1 & 1
\end{array}
$$
Since $\Pic(S)\cong\Pic(S_p)$ and $\mS_p=S_p$,
all the above relations also hold in $\Pic(\mS_p)$.
Hence, in order to prove $\rho(\mS)=34$ using Technique \ref{tech-3},
it suffices to show that neither of the above generators lifts to $\mS$.
For the first one we have already seen this above (using Corollary \ref{cor:2lines}).
For the other two, this would amount to studying quartics $V$ which are sums of two skew quadrics.
Here the main obstruction is that since $V^2=-8$, 
the effectiveness argument from (the proof of) Proposition \ref{prop:lift}
does not apply.
Instead we will argue geometrically with the quotient surfaces $\mS/\imath_1$ and $\mS/\imath_3$.
Abstractly, each quotient can be seen to have $p_g=1$ since both invariant subspaces
$H^{2,0}(\mS)^{\imath_1^*}$ and $H^{2,0}(\mS)^{\imath_3^*}$ are 1-dimensional.
In accordance, the resolutions of the quotient surfaces turn out to be K3 surfaces, 
and quite similar ones indeed.

\subsubsection{1st K3 quotient}
\label{sss:1st}

For starters, we exhibit a birational model of the quotient $X = S/\imath_1$ in the affine chart $x=1$
and the invariant coordinates $v=y+z, u=yz$:
\begin{eqnarray}
\label{eq:1stK3}
X:\;\;\; uw^3+u(v^2-2u)+w=0.
\end{eqnarray}
There are several (equivalent) ways to compactify this affine model to a K3 surface,
for instance as a resolution of a quartic in $\PP^3$.
In fact, $X$ is then seen as a Delsarte quartic, and the methods of \ref{ss:del} show
that $\rho(X)=18$.
For our purposes, it will be most beneficial to interpret $X$ as an elliptic surface over $\PP^1_w$.
Directly, we can regard \eqref{eq:1stK3} as a cubic in $u,v$
with two rational points at $\infty$ (with $u=0$ resp. $v=0$)
and another rational point $(u,v)=(-2/w^2, \sqrt{-2}/w)$ given by the image curve of $Q_1$.
Homogenising \eqref{eq:1stK3} by the additional variable $s$, we immediately obtain a cubic in Weierstrass form.
Normalising and twisting to make the above section rational over $\Q$,
we arrive at
\begin{eqnarray}
\label{eq:1st-2}
X:\;\;\; v^2 = s(s^2+2w^3s-8w).
\end{eqnarray}
This has reducible fibres of type $I_{10}^*$ at $w=\infty$ and $III$ at $w=0$
and a two-torsion section $(0,0)$.
Together with the section $P=(-2w^3,4w^2)$ of height $5/2$ induced by $Q_1$,
these curves generate $\Pic(X)$ of discriminant $-5$.

Deforming $S$ to $\mS$, the above computations go through as before
with the affine deformation factor $pu$.
The resulting Weierstrass form for $\mX=\mS/\imath_1$ reads
\[
\mX:\;\;\; v^2 = s(s^2+2(w^3-p)s-8w).
\]
Here two-torsion section and $I_{10}^*$ fibre at $\infty$ are preserved
while the type $III$ fibre at $w=0$ deforms to Kodaira type $I_2$.
Thus $\rho(\mX)\geq 17$ by the Shioda-Tate formula as should be,
and $\rho(\mX_p)=\rho(X_p)=18$ by construction.
It remains to rule out $\rho(\mX)=18$.
To this end, we will show that the section $P$ on $X$, and thus on $\mX_p$, does not lift to $\mX$ and 
apply Corollary \ref{cor:lift}.
By the theory of elliptic surfaces with section
(as recorded in the Mordell-Weil lattices \cite{ShMW}),
any divisor lifting $P$ would necessarily induce a unique section  of exactly the same shape (over $\Q$!).
Thus we can check explicitly whether $P$ lifts to any given order from $\mX_p$ to $\mX$.
A straight-forward computation involving only linear equations 
reveals that lifting works up to order $p^4$, but not anymore for $p^5$.
Thus $\rho(\mX)=17$ as claimed.
(Alternatively one could try to solve directly for a section of the prescribed shape on $\mX$ over $\C$,
but the system of equations becomes quite complicated.)
In consequence the quartic $Q_1+Q_2$ cannot lift from $\mS_p$ to $\mS$,
i.e.~ the second $G$-module does not lift.

\subsubsection{2nd K3 surface}
\label{sss:2nd}

For the second quotient surface $\mS/\imath_3$, 
a completely analogous argument applies.
In fact, the equations in terms of the invariant coordinates $u=yz, v=y-z$ 
show a striking similarity to those of $\mS/\imath_1$ with only one sign changed:
\[
\mS/\imath_3:\;\;\; uw^3+u(v^2+2u)+w=pu.
\]
In particular, the central fiber $S/\imath_3$ is $\bar\Q$-isomorphic to $X$ (with Picard number 18),
and the same computations with some sign adjustments prove that $\rho(\mS/\imath_3)=17$.
Equivalently, the third $G$-module does not lift from $\mS_p$ to $\mS$.

\subsubsection{Summary}
We have proved that neither of the three irreducible $G$-modules
comprising $\Pic(\mS_p)/L$
lifts from $\mS_p$ to $\mS$.
Hence we deduce that $\rho(\mS)=34$ by Technique \ref{tech-3}.

\subsection{Picard number 38}
\label{ss:38}

Up to isomorphism, there is a
unique Delsarte quintic with Picard number 39 by \cite{S-quintic}.
It is possible to derive quintics with Picard number 38 from this quintic,
but the techniques which we employ
do not exclusively use arithmetic deformations,
and the example we found are not defined over $\Q$.
Therefore we decided to postpone the detailed treatment of these interesting quintics
to a future occasion and pursue a different approach here.
The starting point is a Delsarte surface of Picard number 41:
\[
S:\;\;\; zw^4+xy^4-xyz^3+x^3yz = 0.
\]
From the dominating Fermat surface of degree 24,
one deduces $\rho(S)=41$.
The following curves generate a finite index sublattice of $\Pic(S)$
of discriminant $2^43^6$:
\begin{itemize}
\item
the 4 lines $\ell_{x,z}, \ell_{x,w}, \ell_{y,z}, \ell_{y,w}$ where the two coordinates in the subscript vanish;
\item
the 4 lines at $\{z-x=w^4+y^4=0\}$;
\item
the 4 lines at $\{z+x=w^4-y^4=0\}$;
\item
the exceptional curves at the $A_{15}$ singularity at $[1,0,0,0]$;
\item
the exceptional curves at the six $A_3$ singularities at $[0,0,1,0], [1,0,\pm 1,0]$ and $[0,1,r,0] \; (r^3=1)$.
\end{itemize}
In order to apply Technique \ref{tech-2},
we pick a prime $p\equiv 1\mod 24$,
rescale the equation by some $p$-adic integer $a\in\Q$ such that $u^4-a$ is irreducible in $\F_p[u]$
and deform preserving singularities and the first 8 lines:
\[
\mS:\;\;\; zw^4+axy^4-xyz^3+x^3yz = pxyzw(z-x).
\]
By construction, one has $\rho(\mS)\geq 38$, and we claim that this is in fact an equality
(generically as well as for specific choices of $p$ and $a$).
To see this, note that $\rho(\mS_p)=\rho(S_p)=41$,
where the four lines other than $\ell_{x,z}$ at $\{x+z=0\}$ form a single $\Gal(\F_{p^4}/\F_p)$-orbit 
on $S_p=\mS_p$.
As in \ref{ss:2}, the assumption $\rho(\mS)>38$ implies that some quadric comprising two of these lines
lifts from $\mS_p$ to $\mS$.
Necessarily, this quadric sits inside some hyperplane $H_\lambda=\{x=\lambda z\}$ containing $\ell_{x,z}$.
Here $\lambda\neq 0,1,\infty$ by construction.
We study the pencil 
of residual quadrics
\begin{eqnarray}
\label{eq:38-2}
C_\lambda:\;\;\; w^4+\lambda a y^4+\lambda (\lambda^2-1)yz^3 - p \lambda(1-\lambda)yz^2w = 0.
\end{eqnarray}

\begin{Lemma}
Over $\Q$, the pencil $\{C_\lambda\}$ has generically no reducible members outside $\lambda=0,1,\infty$.
\end{Lemma}

\begin{proof}
We start by computing the singular members of the pencil $\{C_\lambda\}$.
To this end, we factor the partial derivative of \eqref{eq:38-2} with respect to $z$:
\[
\lambda(\lambda-1)yz(3(\lambda+1) z+2pw) = 0.
\]
The first two solutions $y=0$ and $z=0$ lead to known singular fibers at $\lambda=0,1,\infty$,
so we concentrate on the third linear factor.
Substituting into the partial derivative of \eqref{eq:38-2} with respect to $w$,
we obtain twice again the solution $w=0$ plus
\[
9(\lambda+1)^2w+p^3\lambda (\lambda-1)y=0.
\]
Substituting both equations into \eqref{eq:38-2},
we deduce that the pencil $C_\lambda$ has singular fibers at the roots of the following degree 8 polynomial in $\lambda$
\begin{eqnarray}
\label{eq:disc-38}
3^9a(\lambda+1)^8-p^{12}\lambda^3(\lambda-1)^4.
\end{eqnarray}
As a linear polynomial in $a$, this is clearly irreducible.
Hence, generically, the singular fibers at these roots do all take the same shape; if they were to be reducible,
then they would contribute at least 8 to the formula from Lemma \ref{lem:st}.
Together with the contributions from the 4 lines at $\lambda=1$ and from the singularities,
this would give generically
\[
\rho(\mS)\geq 2 + 3 + 33 + 8 = 46
\]
with the required contradiction.
\end{proof}

For an explicit irreducible example of \eqref{eq:disc-38} over $\Q$, 
one can choose $a=5, p=73$, for instance.
It follows that the pencil $\{C_\lambda\}$ on the corresponding quintic $\mS$
has no reducible fibres outside $\lambda=0,1,\infty$.
In particular, none of the quadrics in question may lift from $\mS_p$ to $\mS$.
Hence $\rho(\mS)=38$ as required.

\section{Picard numbers 6 through 8}
\label{s:6-8}

So far, we have applied the arithmetic deformations technique exclusively to Delsarte surfaces
because those allow for effective control over the Picard numbers of the reductions.
This method does not apply to Picard numbers 6 through 8
since these are too far away from the next greater Picard number of a Delsarte surface which is 13 (cf.~Section \ref{s:9-12}).
Instead we work out a quintic (at first resembling the Delsarte surface of Picard number 13
without the singularity) with Picard number 9 over $\bar\F_5$ (and $\bar\Q$)
and apply Techniques \ref{tech-1} and \ref{tech-3} to it.
The general idea is to work with a quintic of the shape
\[
S:\;\;\; f(x,y) = zwg(x,y,z,w)
\]
for homogeneous polynomials $f, g$ of degree $3, 5$,
since the ten lines at $zw=0$ generate a sublattice $V\subset\Pic(S)$ of rank 9.
We will prove $\rho=9$ for some specific $S$
by using reduction modulo $p$ and the specialisation embedding \eqref{eq:spec}.
Recall from \ref{ss:spec} 
that an upper bound for the Picard number of $S\otimes\bar\F_p$ is read off
from the characteristic polynomial of Frobenius on $\het{S\otimes\bar\F_p} \, (\ell\neq p)$.
This, in turn, can,
at least in principle,
be computed from point counts over finite fields using Lefschetz' fixed point formula.
However, for a characteristic polynomial of degree $44$, this would require counting as deep as $\F_{p^{22}}$
which is computationally infeasible.
One could try to overcome this by appealing to $p$-adic cohomology techniques (cf.~\cite{AKR}),
but we decided to rather endow $S$ with some additional structure that would allow us to run the computations
directly.

Crucially, we will use (finite order) automorphisms on $S$.
Here it is quite common to use purely non-symplectic automorphism 
(acting faithfully on $H^{2,0}(S)$ with all eigenvalues 
being roots of unity of the same order $n$).
We have explained in \ref{ss:9}
how such an automorphism affects the transcendental lattice.
For the purely non-symplectic automorphism to commute with the Galois group, however, 
the base fields tend to be rather big
(starting from $\F_{11}$ for $n=5$ as in \ref{ss:9}, for instance).
These computational drawbacks can be remedied by considering wild automorphisms.
By definition, these act trivially on $H^{2,0}(S)$, yet they allow
to decompose the cohomology as explored in \cite{S-wild}
while commuting with the action of the Galois group.

\subsection{Wild automorphism}
\label{ss:wild}

Let $p=5$ and consider the quintic surface
\begin{eqnarray}
\label{eq:zwg}
S:\;\;\; y^5-x^4y = zwg(x,z,w)
\end{eqnarray}
 over $\F_5$
for some irreducible homogeneous cubic polynomial $g\in\F_p[x,z,w]$.
Clearly $S$ admits a wild automorphism of order 5:
\[
\varphi: [x,y,z,w] \mapsto [x,y+x,z,w].
\]
Since $\varphi$ and $\Fr_p$ commute,
we can diagonalise their induced action on $\het{S\otimes\bar\F_p}$ simultaneously.
Generally the only $\varphi^*$-invariant class is the hyperplane section.
The orthogonal complement $U$ of $V$, written after a Tate twist as
\[
U = (V\otimes\Q_\ell(-1))^\perp\subset \het{S\otimes\bar\F_p},
\]
decomposes into four 11-dimensional eigenspaces under $\varphi^*$.
Following \cite{S-wild}
the characteristic polynomial of $\Fr_q^* \, (q=p^r)$ on each of them can be computed from counting
$\# \mbox{Fix}((\varphi^j)^*\circ\Fr_q^*)$ for $j=0,\hdots,4$;
here it is computationally very advantageous
that each fixed point set can be computed directly over $\F_q$.
With these preparations, it should be possible 
to compute the characteristic polynomials of $\Fr_p^*$
on each eigenspace from point counts as deep as $\F_{p^6}$
with machine help and enough patience,
but this still requires implementing Poincar\'e duality and solving a system of 22 non-linear equations
in 20 variables (analogous to what we will achieve in smaller dimension in \ref{ss:9-5}).
We omit the details here and rather specialise to a situation that is yet preferable from our view point
of arithmetic deformations.

\subsection{Extra involution}

We shall now endow the quintic $S$ with even more structure by postulating
that the cubic polynomial $g$ be symmetric in $z,w$.
On $S$, this symmetry induces the involution
\[
\imath: [x,y,z,w] \mapsto [x,y,w,z]
\]
which commutes with both, $\varphi$ and $\Fr_p$.
The $\imath^*$-invariant part of cohomology can be studied
through the quotient surface $S/\imath$,
or rather the minimal resolution $X$ of the $A_1$ singularity resulting from the fixed point at $[0,0,1,-1]$.
The 5 invariant sums of lines together with the exceptional curve
give a sublattice $V_+\subset\Pic(X)$ of rank 6,
so
\[
\rho(X)\geq 6.
\]
In fact, $X$ is a K3 surface.
To see this, express it 
affinely (in the chart $x=1$) as a double sextic in the invariant coordinates $y, u=z+w, v=zw$:
\[
X: \;\;\; y^5-y = v \hat g(u,v).
\]
This representation also confirms that the wild automorphism $\varphi$ descends to $X$.
In particular, this means that the arguments from \ref{ss:wild} about decomposing cohomology into
eigenspaces carry over to $X$.
In particular, we can compute the characteristic polynomial of $\Fr_p^*$ on the $\varphi^*$-eigenspaces inside
\[
U_+ = (V_+\oplus\Q_\ell(-1))^\perp\subset\het{X\otimes\bar\F_p}
\]
from point counts over $\F_p,\hdots,\F_{p^4}$.
This puts us in the position to work towards quintics of Picard number 6 to 8.

\subsection{A K3 surface with Picard number 6 over $\bar\F_5$}
\label{ss:K3-6}

We shall now specialise the above discussion to the cubic polynomial
\[
g = x^3+(x-z-w)(z^2+zw+w^2).
\]
One directly verifies that the quintic $S$ is indeed smooth, and so is the K3 surface $X$.
We start by computing the Picard number of $X$.
Note that $\hat g=1-u^3+vu+u^2-v$.

\begin{Lemma}
\label{lem:9-6}
Over $\bar\F_5$, the K3 surface $X$ has $\rho(X)=6$.
\end{Lemma}

\begin{proof}
As explained in \ref{ss:wild},
we can compute $\# \mbox{Fix}((\varphi^j)^*\circ\Fr_{p^i}^*)$ for $j=0,\hdots,4, i=1,\hdots,4$,
to derive the characteristic polynomials of $\Fr_p^*$ on the eigenspaces of $\varphi^*$ inside $U_+$.
Overall we obtain the characteristic polynomial $\chi_p^+(\lambda)$ of $\Fr_p^*$ on $U_+$:
\begin{eqnarray*}
\chi_p^+(\lambda) & = & 
\lambda^{16}+2p^3\lambda^{13}+p^3\lambda^{12}+6p^4\lambda^{11}
+2p^6\lambda^{10}+4p^6\lambda^9+9p^7\lambda^8\\
&& \;\;\; +4p^8\lambda^7+2p^{10}\lambda^6
+6p^{10}\lambda^5+p^{11}\lambda^4+2p^{13}\lambda^3+p^{16}.
\end{eqnarray*}
This polynomial is irreducible over $\Z$ and not of cyclotomic shape
(i.e.~arising from a cyclotomic polynomial in $\lambda$ 
via the substitution $\lambda\mapsto\lambda/p$ and clearing denominators).
Hence $U_+$ cannot contain any algebraic classes,
and the lower bound $\rho(X)\geq 6$ is attained.
\end{proof}

\subsection{A quintic with Picard number 9 over $\bar\F_5$}
\label{ss:9-5}

We continue with the quintic $S$ and $p=5$.
Recall that we are concerned with the 44-dimensional Galois representation $U\subset\het{S\otimes\bar\F_p}$.
Here the induced action of $\imath^*$ decomposes $U$ into $(\pm 1)$-eigenspaces.
Lemma \ref{lem:9-6} shows
that the 16-dimensional eigenspace
\[
U^{\imath^*=1} = U_+
\]
does not contain any algebraic classes.
Hence it remains to consider the $28$-dimensional eigenspace $U^{\imath^*=-1}$ 
in order to prove the following result:

\begin{Proposition}
\label{prop:9}
Over $\bar\F_5$, the quintic surface $S$ has $\rho(S)=9$.
\end{Proposition}

\begin{proof}
We want to compute the characteristic polynomial $\chi_p^-(\lambda)$
of $\Fr_p^*$ on $U_-$,
or rather on the eigenspaces under $\varphi^*$.
We attack this problem by computing the first 4 coefficients
by point counts over $\F_p,\hdots,\F_{p^4}$ as explained in \ref{ss:wild}
and subtracting the contribution from $U_+$ computed in \ref{ss:K3-6}.
This leaves three coefficients over $\Z[\zeta] \, (\zeta\in\mu_5)$ and their conjugates;
that is, over $\Z$ there are 12 coefficients missing from $\chi_p^-(\lambda)$.
To determine these coefficients,
we use Poincar\'e duality for $U_-$;
the functional equation for $\chi_p^-(\lambda)$ gives
14 equations in the 12 unknowns with rather big coefficients over $\Z$.
Note that these equations are not linear, but have degree up to 4,
so it is a highly non-trivial task to solve the system of equations.

We use a $2$-adic method building on the fact that we know a priori the existence of a solution over $\Z$.
First we determine, by an exhaustive search, all solutions over $\F_2^{12}$;
there are 4 in number.
Then we iterate to increase the $2$-adic accuracy;
this principle  resembles  the $2$-adic Newton iteration in \cite[Algorithm 9]{ES},
but presently the situation is slightly different since
at each step increasing the $2$-adic accuracy by a single power of $2$
we have to solve a system of  linear equations of size $14\times 12$ over $\F_2$.
In particular, the solution space need not be 0-dimensional,
and at each step we have to check whether any solution in the affine linear space
allows for increasing the accuracy.

For three of the initial solutions modulo $2$,
this interation terminates after at most four steps (accuracy $2^5$).
Meanwhile for the remaining initial spolution in $\F_2^{12}$, we find 2-dimensional affine solution spaces at each step,
but from accuracy $2^6$ onwards,
there is always a unique point in the solution space
which allows for an increase of accuracy.
We stopped the computations at accuracy $2^{23}$.
Using the Weil conjectures and the bounds which they impose on the coefficients of the characteristic polynomials
(roughly estimated by considering separately real part and imaginary part),
this accuracy suffices to read off the coefficients over $\Z$.
We express the solution as follows (related to the Tate twist $U_-(1)$):
\begin{eqnarray*}
p^{-25} \chi_p^-(5\lambda) & = & 
125\,{\lambda}^{28}+50\,{\lambda}^{26}+25\,{\lambda}^{25}+5\,{\lambda}^{24}+30\,{\lambda}^{23}+125\,{\lambda}^{22}
\\
&&
\mbox{} -85\,{\lambda}^{21}+171\,{\lambda}^{20}  -89\,{\lambda}^{19}
-65\,{\lambda}^{18}-105\,{\lambda}^{17} -24\,{\lambda}^{16}
\\
&&
\mbox{}
-165\,{\lambda}^{15} +129\,{\lambda}^{14}-165\,{\lambda}^{13}-24\,{\lambda}^{12}
-105\,{\lambda}^{11}\\
&&
\mbox{}
-65\,{\lambda}^{10}-89\,{\lambda}^{9}+171\,{\lambda}^{8}
-85\,{\lambda}^{7}+125\,{\lambda}^{6}+30\,{\lambda}^{5}
\\
 &&
\mbox{}+5\,{\lambda}^{4}
+25\,{\lambda}^{3}+50\,{\lambda}^{2}+ 125.
\end{eqnarray*}
One checks that the above polynomial is irreducible over $\Q$.
Since it is  not of cyclotomic shape,
$U_-$ cannot contain any algebraic classes, either.
In conclusion, neither does $U$,
and the lower bound $\rho(X)\geq 9$ is attained.
\end{proof}

\subsection{Picard number 8}
\label{ss:8}

We deform the surface $S$ over $\Q$ such that 8 lines are preserved:
\[
\mS:\;\;\; y(y^2+x^2)(y^2-x^2+pw^2) = zw(x^3+(x-z-w)(z^2+zw+w^2)).
\]
Thus $\rho(\mS)\geq 8$. On the other hand, 
\[
\rho(\mS)\leq\rho(\mS_p)=\rho(S_p)=9,
\]
so we can apply Corollary \ref{cor:2lines} to the lines $\{y\pm x=z=0\}$
to deduce that the neither line lifts to $\mS$.
Hence $\rho(\mS)=8$ by Corollary \ref{cor:lift}.

\subsection{Picard number 7}
\label{ss:7}

We continue by deforming $S$ preserving 6 lines and the involution $\imath$:
\[
\mS:\;\;\; y(y^2+x^2)(y^2-x^2+p(w+z)^2) = zw(x^3+(x-z-w)(z^2+zw+w^2)).
\]
Exactly as in \ref{ss:8}, we obtain
\[
7\leq\rho(\mS)<\rho(\mS_p)=\rho(S_p)=9,
\]
since neither of the lines
\[
\ell_\pm = \{y\pm x=z=0\},\;\;\; \ell'_\pm = \{y\pm x=w=0\}
\]
lifts to $\mS$ by Corollary \ref{cor:2lines}.
In order to infer that $\rho(\mS)=7$,
we employ Technique \ref{tech-3} with $G=\{\id,\imath\}$. 
If we had $\rho(\mS)>7$,
then this would mean that either the 
$\imath^*$-invariant divisor $\ell_++\ell'_+$
or the anti-invariant divisor $\ell_+-\ell'_+$ lifts to $\mS$.
We start by ruling out the latter assumption.
Equivalently, the quadric $\mathfrak Q=\ell_++\ell'_-$ lifts.
By Proposition \ref{prop:lift}, the lift is a unique quadric $Q$
with $Q^2=-6$.
But then, by inspection of the self-intersection number,
$Q$ decomposes into 2 skew lines which necessarily lift $\ell_+$ and $\ell'_-$,
contradiction.

It remains to study the  divisor $\ell_++\ell'_+$
which we assume to lift to $\mS$ (uniquely by Proposition \ref{prop:lift}, hence over $\Q$).
Equivalently,
the tritangent lines
\[
\mathfrak L_\pm=\{y\pm1=v=0\}
\]
deform from the K3 quotient $X$ to
\[
\mX:\;\;\;  y(y^2+1)(y^2-1+pu^2) = v(1-u^3+vu+u^2-v).
 \]
Here we can argue similarly to  Section \ref{s:lift},
but the reasoning is greatly simplified by the fact
that for any divisor $D$ on a K3 surface with $D^2=-2$,
either $D$ or $-D$ is effective by Riemann-Roch.
Applied to the lifting problem, this means that an effective divisor $\D$ with $\D^2=-2$
lifts to an effective divisor if it lifts at all,
and the lift is unique by the negative self-intersection number and the conservation of the degree.

Presently,
the quadric $\mathfrak L_++\mathfrak L_-$ on $\mS_p$
lifts uniquely to the quadric 
\[
Q=\{y^2-1+pu^2=v=0\}\subset\mS.
\]
But then if either line $\mathfrak L_\pm$ were to lift to $\mS$,
then it would automatically lift to a component of $Q$.
However, $Q$ is irreducible by construction.
This gives the required contradiction and concludes the proof that $\rho(\mS)=7$ 
(and $\rho(\mX)=5$).

\subsection{Towards Picard number 6}
\label{ss:towards6}

The case of Picard number 6 is yet different from the previous ones
because it is non-trivial to endow the quintics from the previous subsections 
with more symmetry without forcing the Picard number to actually go up
(similar to what goes on for K3 surfaces with symplectic automorphisms
as studied by Nikulin \cite{Nikulin}).
We will overcome this by first deforming a singularity
(the only instance within this paper)
and at the final step using point counts at a second prime even more expansive than in the previous  subsections.

The general idea persists though:
we work with a quintic $S$ with an involution $\imath$
such that the quotient $S/\imath$ has a K3 surface $X$ as resolution.
Then we try to control the transcendental eigenspaces
$T_+ = T(X)$ and $T_-$ inside $H^2(S)$ separately.
Again we will endow the surfaces with a wild automorphism over $\F_5$
so that we can prove $\rho(S\otimes\bar\F_5)=9$.
This severely limits the possible structure on $S$.
Our goal is to deform away two $\imath^*$-invariant divisors
and one $\imath^*$-anti-invariant divisor on $S$.
This was impossible (in a provable way) for the quintics from \ref{ss:wild}
since there one typically first deforms an anti-invariant divisor before going for an invariant divisor.
Instead we start with a quintic with 4 lines and an $A_4$-singularity.

\subsection{Another quintic with Picard number 9}
\label{ss:another}

Consider the quintic
$S$ over $\F_5$
taking a slightly different shape compared to \eqref{eq:zwg},
\[
S:\;\;\; y^5-x^4y = 
(z+w)\, g(x,z,w)
\]
 given by the  polynomial
\[
g(x,z,w) = (z-w)^2(z^2+zw+w^2+x(z+w))-x^2(z^2-zw+w^2)+x^3(z+w)-2x^4.
\]
It is equipped with 
\begin{itemize}
\item
an involution $\imath$ interchanging $z, w$,
\item
5 $\imath$-invariant lines at $z+w=0$
and 
\item
4 exceptional curves above the $A_4$-singularity at $[0,0,1,-1]$.
\end{itemize}

\begin{Proposition}
\label{prop:6-5}
The quintic $S$ has Picard number $9$ over $\bar\F_5$.
\end{Proposition}

\begin{proof}
From the above curves, it is evident that $\rho(S)\geq 9$.
It remains to show that $\rho(S\otimes\bar\F_5)=9$.
We proceed as in \ref{ss:wild} - \ref{ss:9-5}.
For the K3 surface $X=\widetilde{S/\imath}$,
we compute the characteristic polynomial $\chi_+(\lambda)$ of Frobenius
on the 16-dimensional complement $U_+$ of the known (algebraic) subspace of $H^2_{\text{\'et}}(X\otimes\bar\F_5,\Q_\ell(1))$
(computed as in \ref{ss:K3-6}):
\begin{eqnarray*}
\chi_+(\lambda) & = &
T^{16}-T^{15}+2T^{14}-T^{13}+T^{12}-T^{10}+2T^9-\frac{14}5T^8\\
&& \;\;\; +2T^7-T^6+T^4-T^3+2T^2-T+1.
\end{eqnarray*}
Since $\chi_+$ is irreducible over $\Q$ and not integral, so neither cyclotomic, we 
infer $\rho(X)=6$ (over $\bar\F_5$ and thus also over $\C$)
and rank$(T_+)=16$.

We continue by proving rank$(T_-)=28$ as in \ref{ss:9-5}.
It suffices to show that the relevant 28-dimensional eigenspace 
$U_-\subset H^2_{\text{\'et}}(S\otimes\bar\F_5,\Q_\ell(1))$
does not contain any algebraic classes.
To this end, we compute the characteristic polynomial $\chi_-(\lambda)$ as
\begin{eqnarray*}
5^3 \chi_-(\lambda) & = & 125\,\lambda^{28}-375\,\lambda^{27}+825\,\lambda^{26}-1250\,\lambda^{25}+1405\,\lambda^{24}-1095\,\lambda^{23}
\\
&&
+130\,\lambda^{22}+1130\,\lambda^{21} -2339\,\lambda^{20}+2862\,\lambda^{19}-2384\,\lambda^{18}
\\
&&
 +1025\,\lambda^{17} +805\,\lambda^{16}-2313\,\lambda^{15}+2914\,\lambda^{14}-2313\,\lambda^{13}
\\
&&
+805\,\lambda^{12}+1025\,\lambda^{11}-2384\,\lambda^{10}+2862\,\lambda^9-2339\,\lambda^8+1130\,\lambda^7
\\
&&
+130\,\lambda^6-1095\,\lambda^5+1405\,\lambda^4-1250\,\lambda^3+825\,\lambda^2-375\,\lambda+125.
\end{eqnarray*}
As before, we find that $\chi_-$ is irreducible over $\Q$ and not cyclotomic.
Hence $U_-$ cannot contain any algebraic classes (neither over $\bar\F_5$ nor over $\C$)
and $T_-=U_-$ has dimension 28 as claimed.
\end{proof}

\begin{Remark}
By inspection of the characteristic polynomial of Frobenius,
the K3 surface $X$ over $\F_5$ has height $8$.
We are not aware of another explicit example of this height in the literature.
\end{Remark}

\subsection{1st Deformation}
\label{ss:def-1}

It was already indicated in the proof of Proposition \ref{prop:6-5}
that
any lift $\mS$ of $S$ from $\F_5$ to $\C$ has Picard number $\rho(\mS)\leq 9$.
Using the lifting obstructions for divisors from Section \ref{s:lift}, one can easily do better.
In practice, we will deform the equation in two directions:
on the one hand, we will deform two pairs of lines to irreducible quadrics.
By Corollary \ref{cor:2lines}, neither of the 4 lines lifts, so $\rho(\mS)\leq 8$ by Corollary \ref{cor:lift}.
However, in order to reach $\rho=6$, we will have to control the lifting of linear combinations of all 4 lines;
for this, we will use reduction modulo a second prime 
and the Galois action on $\NS$ with a view towards Technique \ref{tech-2} in \ref{ss:6-3}.

On the other hand, we will deform the singularity from $A_4$ to $A_3$.
Recall from \ref{ss:sing} that this still allows for good reduction:
the first blow-up of the singularity results in exceptional components $E_1, E_4$;
after the second blow-up, the exceptional locus is given by  a quadric $Q$
which is either irreducible or decomposes into two rational curves $E_2, E_3$.
That is, in the deformation $\mS$, $E_2+E_3$ deforms to $Q$.

%

The crucial property for us is that while $\imath$ leaves all 5 lines invariant,
it exchanges the exceptional curves above the $A_4$ singularity of $S$.
That is, by Technique \ref{tech-3} we can decide on the lifting of the anti-invariant divisor $E_2-E_3$
independently of the lines.
Since $E_2+E_3=Q$ lifts by assumption,
lifting $E_2-E_3$ is equivalent to lifting $E_2, E_3$ by Theorem \ref{thm:prim}.
Hence we infer from Lemma \ref{lem:sing}:

\begin{Lemma}
\label{lem:rk29}
If the $A_4$ singularity on $S$ deforms to $A_3$ on $\mS$,
then 
\[
\rank\; T_-(\mS)\geq 29.
\]
\end{Lemma}

%
%
%

\subsection{Second Reduction}
\label{ss:6-3}

Our aim now is to deform $S$ such that $\rank\, T_+=18$
(while \ref{ss:def-1} only gives $\rank\, T_+\geq 17$).
To this end, we need to control the deformations of the lines on $S$.
Since two of them are conjugate over $\Q(\sqrt{-1})$,
they lend themselves to Technique \ref{tech-2},
but we have to work with the reduction modulo a different prime $p$ such that $\sqrt{-1}\not\in\F_p$;
in what follows we will take $p=3$. 
Specifically, consider the quintic $S'$ over $\F_3$ given as in \ref{ss:another}
by
\[
S':\;\;\; y^5-x^4y = 
(z+w)\, g'(x,z,w),
\]
but with a modified  polynomial
\[
g'(x,z,w) = (z-w)^2(zw+x(z+w))+x^2(z^2+w^2). %
\]
For our purposes it suffices to work with the K3 quotient $X'=\widetilde{S'/\imath}$,
since we are in essence only concerned with the eigenspace $T_+ = T(X')$.
This is of great relevance not only because it reduces the complexity of the computations,
but also because the quintic $S'$ attains 8 additional nodes over $\F_9$ 
whose exceptional curves are contracted in the quotient $X'$.
Affinely $X'$ can be given as a double sextic
\begin{eqnarray}
\label{eq:X'}
X':\;\;\; v^2 = y^5-x^4y+x^4+x^3+x^2+x+1.
\end{eqnarray}
From the invariant cycles on $S'$ and the isolated fixed point of $\imath$ at $[0,0,1,-1]$, 
we obtain a 6-dimensional algebraic subspace of $H^2_{\text{\'et}}(X'\otimes\bar\F_3,\Q_\ell(1))$.
We denote its orthogonal complement by $U'$ and compute the characteristic polynomial $\chi'(\lambda)$
of Frobenius on $U'$ by point counts over $\F_3,\hdots,\F_{3^8}$
using Lefschetz' fixed point formula and Poincar\'e duality:
\begin{eqnarray}
\label{eq:chi'}
\chi'(\lambda) = (\lambda^8-\lambda^4+1)\left(\lambda^8+\frac 23\lambda^4+1\right).
\end{eqnarray}
Either factor is irreducible over $\Q$. The second factor, being non-integral, corresponds to
a non-algebraic subspace of $H^2_{\text{\'et}}(X'\otimes\bar\F_3,\Q_\ell(1))$.
Meanwhile, the first factor corresponds to an irreducible Galois-submodule 
$M'\subset H^2_{\text{\'et}}(X'\otimes\bar\F_3,\Q_\ell(1))$ of dimension 8 
which is algebraic by the Tate conjecture (cf. \cite{Madapusi}),
so $\rho(X'\otimes\bar\F_3)=14$.
The next section continues 
with a simultaneous deformation of $S$ and $S'$ to exhibit a quintic with Picard number 6.

\begin{Remark}
Even without assuming the validity of the Tate conjecture, we could proceed 
to exhibit a quintic $\mS$ of Picard number 6 as follows:
if the Galois module $M'\subset H^2_{\text{\'et}}(X'\otimes\bar\F_3,\Q_\ell(1))$ 
corresponding to the first factor of \eqref{eq:chi'}
is algebraic, then continue as in \ref{ss:6}.
If $M'$ were not algebraic, i.e. if the Tate conjecture were  not to be valid,
then this would mean $\rho(X'\otimes\bar\F_3)\leq 6$ right away,
and we would not have deform away any divisor classes corresponding to $M'$ anyway.
\end{Remark}

\subsection{Picard number 6}
\label{ss:6}

Consider the quintic $\mS$ over $\Q$ arising  as minimal resolution from the following polynomial:
\[
\mS:\;\;\; y(y^2-x^2+15(z-w)^2)(y^2+x^2+15(z-w)^2) = (z+w) \, \tilde g(x,z,w),
\]
where 
\[
\tilde g =  (z-w)^2(6z^2+zw+6w^2+x(z+w))+x^2(4z^2+6zw+4w^2)+6x^3(z+w)+3x^4+15y^4.
\]
By definition, $\mS$ has three exceptional curves above an $A_3$-singularity at $[0,0,1,1]$
while the plane $\{z+w=0\}$ decomposes into a line $\ell_0$ and two conics $Q_1, Q_2$. 
Hence $\rho(\mS)\geq 6$, and we claim that in fact equality is attained:

\begin{Theorem}
\label{thm:6}
The complex quintic $\mS$ has Picard number 6.
\end{Theorem}

\begin{proof}
We will use reduction modulo $3$ and $5$.
For starters, note that 
\[
\mS\otimes\F_5 = S\otimes\F_5.
\]
Hence $\rho(\mS)\leq 9$ by Proposition \ref{prop:6-5}.
More precisely, using the structure of the arithmetic deformation $\mS$, we have 
\begin{eqnarray*}
\label{eq:rk29}
\rank\, T_-(\mS)\geq 29
\end{eqnarray*}
by Lemma \ref{lem:rk29} (since the $A_4$ singularity is deformed to type $A_3$),
and $\rank\, T_+(\mS)\geq 17$ as argued in \ref{ss:def-1} (since some lines do not lift from $S$ to $\mS$).
That is, 
\[
6\leq \rho(\mS)\leq 7.
\]
In order to see that the first equality is attained, we use reduction modulo $3$ to prove that $\rank\,T_+(\mS)=18$.
Here
$
\mS\otimes\F_3 = S'\otimes\F_3,
$
but for computing the rank of $T_+(\mS)$ it suffices to work with the quotient K3 surfaces
\[
\mX\otimes\F_3 = X'\otimes\F_3
\]
where $\mX$ is a minimal desingularisation of ${\mS/\imath}$ and $T_+(\mS) = T(\mX)$.
To apply the arithmetic deformation technique,
we first verify that $\mX$ has indeed good reduction at $3$.
To this end, it is convenient to work with the model in weighted projective space $\PP[1,1,1,2]$
arising directly from $\mS$ in the invariant coordinates $u=z+w, v=zw$
(e.g.~homogenise \eqref{eq:X'} accordingly).
Here the deformation of the rational curves mapping down $2:1$ from the quintic 
(5 lines $\ell_0,\hdots,\ell_4$ on $S'$ resp.~the line $\ell_0$ and 
two conics $Q_1, Q_2$ on $\mS$)
is still visible in the hyperplane $u=0$.
Another independent class $E$ in $\NS(X')$ resp.~$\NS(\mX)$ 
arises from the isolated fixed point of $\imath$ at $[0,0,1,-1]$;
it is located at the singularity of the ambient weighted projective space.
We sketch the curves in the following figure:

\begin{figure}[ht!]
\setlength{\unitlength}{.25in}
\begin{picture}(15,6)(0,0)

\put(11,0){\line(0,1){6}}

\put(10.3,0){$E$}

\multiput(10,1)(0,1){5}{\line(1,0){5}}
\put(15.2,4.8){$\ell_0$}
\put(15.2,3.8){$\ell_1$}
\put(15.2,2.8){$\ell_2$}
\put(15.2,1.8){$\ell_3$}
\put(15.2,.8){$\ell_4$}
  

\put(1,0){\line(0,1){6}}

\put(0,5){\line(1,0){5}}

\put(0.3,0){$E$}

\put(-.7,4.8){$\ell_0$}

\put(2.5,.5){$Q_1$}

\put(4.45,.5){$Q_2$}

\put(7,3){$\leadsto$}

\qbezier(5,6)(2.5,3.5)(5,1)
\qbezier(2.5,6)(5,3.5)(2.5,1)
%
%
%
%
%
%
%
%
%

\end{picture}
\end{figure}

In terms of divisors on the K3 surfaces the degeneration of the conics involves the exceptional curve $E$:
\[
 Q_1=\ell_1+\ell_2+E,\;\;\; Q_2=\ell_3+\ell_4+E.
 \]
In summary, this affirms that $\rho(X')\geq 6$ resp. $\rho(\mX)\geq 4$, and that $\mX$ has good reduction at $3$.

We are now in the position to apply the arithmetic deformation technique to $\mX$ considered as lift of $X'_3$.
In detail, we want to utilise the Galois action on $\Pic(X'_3)$ in Technique \ref{tech-2}.
To this end, let the rank 4 sublattice $L\subset\Pic(X'_3)$ be generated by the classes 
$\ell_0, Q_1, Q_2, E$ lifting to $\mX$.
By \ref{ss:6-3}, the quotient $\Pic(X'_3)/L$ decomposes into three irreducible Galois modules over $\F_3$.
We will now show that neither of them lifts to $\mX$.
The rank 8 Galois-module $M'$ cannot lift to $\mX$ for rank reasons,
since otherwise the Picard number would exceed the bound $\rho\leq 7$ imposed by reduction modulo $5$.
The other two modules are generated by $\ell_1-\ell_2$ (Galois-invariant)
resp. $\ell_3-\ell_4$ (anti-invariant under $\Gal(\F_9/\F_3)$).
By Theorem \ref{thm:prim}, lifting either module is equivalent to lifting the two lines themselves 
(since their sums are in $\Pic(\mX)$).
Arguing as before with Riemann-Roch,
a lift $\mathfrak L$ would necessarily be effective. But then
$\mathfrak L.Q_i=-1$ implies that $\mathfrak L$ is a component of $Q_i \;(i=1$ or $2)$.
This gives a contradiction, since the quadrics in question
 ($y^2\pm x^2+60z^2$) are clearly irreducible over $\C$.
Thus, neither irreducible Galois module lifts from $X'_3$ to $\mX$,
and Technique \ref{tech-2} proves that
\[
\rho(\mX) = 4, \;\;\; \rank\, T(\mX)=18.
\]
On $\mS$, we obtain
\[
\rho(\mS)\geq 6,\;\;\; \rank\, T_+(\mS)=18,\;\;\; \rank\, T_-(\mS)\geq 29.
\]
Comparing with the second Betti number $b_2(\mS)=53$, we infer that the inequalities are in fact equalities.
\end{proof}

\section{Picard numbers by related techniques}
\label{s:other}

In this section we complete the proof of Theorem \ref{thm}.
To this end, we engineer quintic surfaces with Picard numbers $\rho=30,40,42,44$.
Unfortunately there are no quintic Delsarte surfaces 
where our arithmetic deformations technique would apply directly;
the problem usually is that we either change the singularities significantly
or lose control over the deformations of divisor classes.
Instead, we pursue the almost opposite approach
by deforming quintics to increase the Picard number.

\subsection{Picard number 30}
\label{ss:30}

In this section,
we will deform a quintic Delsarte surface to increase the Picard number to 30.
Crucially, we will impose enough automorphisms
to endow the surface with sub-Hodge structures governed by algebraic curves.
In detail, we start with the Delsarte surface given by
\[
S:\;\;\; w^5+y^5+y(x^4+z^4)=0.
\]
From the covering Fermat surface of degree $20$, we read off $\rho(S)=29$.
Up to finite index, $\Pic(S)$ is generated by the exceptional curves over the $A_4$ singularities at
$[1,0,r,0] \, (r^4=-1)$ and the 21 lines at $y=w=0$ and at $w^5+y^5=x^4+z^4=0$.
Below we will use the involutions on $S$ leaving $y,w$ invariant and acting on $x,z$ as
\[
\imath_1(x,z) = (-x,-z),\;\; \imath_2(x,z)=(-x,z),\;\; \imath_3(x,z)=(x,-z).
\]
We shall now deform the surface $S$ retaining 
singularities, lines, involutions, and for convenience, the symmetry in $x,z$.
To this end, let $f\in\Q[y,w]$ be homogenous of degree $5$ and $a\in\Q$ (to be specialised later).
Then the quintic surface
\[
\mS: \;\;\; f(y,w) + y(x^4+ax^2z^2+z^4) = 0
\]
satisfies all the above requirements, so in particular $\rho(\mS)\geq 29$
outside degenerate cases where, for instance, $a=\pm 2$ or $f$ has multiple factors. 
We study $\mS$ through the quotients by the involutions.
From the action on $H^{2,0}(\mS)$ we deduce that there is a  splitting
\[
T(\mS) = T(\mS/\imath_1) \oplus T(\mS/\imath_2) \oplus T(\mS/\imath_3)
\]
as Galois representations or transcendental Hodge structures over $\Q$.
Moreover, $T(\mS/\imath_2)\cong T(\mS/\imath_3)$ by symmetry:
\begin{eqnarray}
\label{eq:T-30}
T(\mS) = T(\mS/\imath_1) \oplus T(\mS/\imath_2)^2.
\end{eqnarray}

\subsubsection{Second quotient}
\label{ss:30-2}

We analyse $\mS/\imath_2$ in the affine chart $y=1$ with $g(w)=f(1,w)$ and invariant coordinate $u=x^2$:
\[
\mS/\imath_2:\;\;\; g(w) + u^2 + auz^2 + z^4 = 0.
\]
Over $\PP^1_w$ this is an isotrivial elliptic fibration 
where the generic fibre has an automorphism $\varphi$ of order $4$ given by $(u,z)\mapsto (-u,iz)$.
Over a finite extension of ${\Q(a)}$, it converts to the Weierstrass form
\begin{eqnarray}
\label{eq:30-2}
\mS/\imath_2:\;\;\; \eta^2 = \xi^3 - g(w) \xi
\end{eqnarray}
with 2-torsion section $(0,0)$ and singular fibres of type $III$ at the roots of $g$ and $III^*$ at $\infty$.
Thus $\rho(\mS/\imath_2)\geq 14$
with equality at the specialisation $S/\imath_2$.
The latter statement can be derived from the purely non-symplectic automorphism
combining the order 4 automorphism $\varphi(\xi,\eta)=(-\xi,i\eta)$ of the generic fibre
with the degree 5 automorphism $w\mapsto \zeta_5 w$ acting on the base;
by \ref{ss:upper}
 this implies $\phi(20)=8\mid\mbox{rank} \, T(S/\imath_2)$, so for a complex K3 surface $\rho(\S/\imath_2)=6,14$.

In order to compute the Picard number at a different specialisation,
we will crucially use that the elliptic  fibration \eqref{eq:30-2} is trivialised by a product of curves
involving the fiber 
\[
E_i: \;\;\; \eta^2 = \xi^3-\xi
\]
and the genus 6 curve given affinely by
\[
C:\;\;\; t^4 = g(w).
\]
Here $\mS/\imath_2$ is birational to the quotient $(C\times E_i)/\langle\Psi\rangle$
where $\Psi = \psi\times\varphi$ and $\psi\in\Aut(C)$ sends $t$ to $it$.
By the K\"unneth formula, $T(\mS/\imath_2)$ is the irreducible transcendental sub-Hodge structure of
a rank 8 Hodge  structure $T_2$ of weight $2$ on $C\times E_i$ which splits over $\Q(i)$ 
into the following eigenspaces:
\begin{eqnarray*}
T_2 & = & 
(H^1(C)^{\psi^*=i}\otimes H^1(E_i)^{\varphi^*=-i}) \oplus
(H^1(C)^{\psi^*=-i}\otimes H^1(E_i)^{\varphi^*=i})\\
&
\subset & H^2(C\times E)^{\Psi^*=1}.
\end{eqnarray*}
For later reference, we note that the above eigenspaces on $H^1(C)$ do indeed live on the Prym curve
of the double covering
\begin{eqnarray*}
C\;\; & \to & C' = \{r^2 = g(w)\}\\
(w,t) &  \mapsto & (w,t^2).
\end{eqnarray*}
In \ref{ss:spec-30} we will choose $g$ and $a$ in such a way that $T_2$ is indeed irreducible of rank 8, so that $\rho(\mS/\imath_2)=14$.

\subsubsection{First quotient}
\label{eq:30-1}

We continue by analysing the quotient $\mS/\imath_1$.
Equivalenly, we can quotient out by (the group generated by)
the order 4 automorphism $\gamma: (x,y,z,w)\mapsto (ix,y,-iz,w)$,
since  by inspection of the action on the holomorphic 2-forms,
\[
T_1 := T(\mS/\imath_1) \cong T(\mS/\gamma).
\]
In the invariant coordinates $u=x^4, v=xz$, the quotient is affinely given by
\[
\mS/\gamma:\;\;\; ug(w) + u^2+ auv^2 + v^4 = 0.
\]
Over a finite extension of $\Q(a)$, we obtain an isotrivial elliptic fibration in Weierstrass form
\begin{eqnarray}
\label{eq:iso-30}
\mS/\gamma:\;\;\;
v^2 = s(s^2-ag(w)s+g(w)^2s).
\end{eqnarray}
Generically, this has 6 singular fibers of type $I_0^*$ at the roots of $g$ and at $\infty$
and full two-torsion over $\Q(\sqrt{a^2-4})$.
In particular, we have $\rho(\mS/\gamma)\geq 26$, 
and at the special Delsarte member $\rho(S/\gamma)=26$ as before.

Again, we can trivialise the above elliptic fibration by a product of curves,
namely by the genus 2 curve $C'$ covered by $C$ and the fiber
\[
E_a:\;\;\; \eta^2 = \xi (\xi^2-a\xi+1).
\]
Comparing dimensions and using the K\"unneth formula, we find generic equality in
\[
T_1:= T(S/\gamma) \subseteq H^1(E_a) \otimes H^1(C').
\]

\subsubsection{Specialisation with Picard number 30}
\label{ss:spec-30}

In order to derive a quintic $\mS$ with $\rho(\mS)=30$,
it suffices in view of \eqref{eq:T-30} to set up our equation in such a way
that $T_1$ degenerates to rank 7
while $T_2$ continues to have rank 8.
Since the second equality holds generically true, 
we start by explaining how to achieve the first degeneration.
Equivalently, we are concerned with the Picard number of $E_a\times C'$
which generically equals 2.
In order to endow the product with an additional algebraic cycle,
we postulate that the Jacobian of $C'$ is isogenous to a product involving $E_a$.
Most elementary, this is achieved by assuming that $C'$ is a double cover of $E_a$.
In the following, we simply spell out this condition: 
after a linear transformation of $\PP^1$, 
the curve $C'$ can be given by the following degree 6 polynomial coming from $E_a$:
\[
(y^2 -r^2) ((y^2-r^2)^2-a(y^2-r^2)+1).
\]
Indeed, after the translation $y\mapsto y+r$, the above equation reads in the coordinate $w=1/y$
\[
g(w):=(1+2rw)(1+4rw+4r^2w^2-aw^2-2arw^3+w^4).
\]
With this choice of $g$,
the Jacobian of $C'$ splits as $E_a\times E_a'$
where $E_a'$ is the elliptic  Prym curve of the double covering $C'\to E_a$.
In particular, we obtain an inclusion of Hodge structures
\[
T_1 \subseteq \mbox{Sym}^2 H^1(E_a) \oplus (H^1(E_a)\otimes H^1(E_a')).
\]
It will follow from the next lemma that this is an equality generically.
We will work with the specialisation at  $a=1, r=1/2$ such that
\[
g = (1+w)(1+2w-w^3+w^4).
\]

\begin{Lemma}
For $g$ as above,
the quintic $\mS$ has $\rho(\mS)=30$.
\end{Lemma}

\begin{proof}
The lemma's statement is equivalent to $T(\mS)$ having rank 23.
We will use \eqref{eq:T-30} to show this.

First it is easy to prove that the transcendental Hodge structure $T_1$ has rank 7 at the given values.
To see this, not that $E_1$ does not have CM (by inspection of the j-invariant), 
so $ \mbox{Sym}^2 H^1(E_1)$ is indeed irreducible of rank 3.
On the other hand, $E_1$ and $E_1'$ are not isogenous,
for instance because the traces of Frobenius do not agree at $p=11$ up to sign.
Hence $H^1(E_a)\otimes H^1(E_a')$ is irreducible of rank 4 which proves the claim.

By \eqref{eq:T-30}, it remains to prove that $T_2$ stays irreducible of rank 8
at the chosen special values.
This will follow at once after we compute the characteristic polynomial of some  Frobenius element on
the 8-dimensional Galois representation
\[
T_2 = (H^1(C)^{\psi^*=i}\otimes H^1(E_i)^{\varphi^*=-i}) \oplus
(H^1(C)^{\psi^*=-i}\otimes H^1(E_i)^{\varphi^*=i}).
\]
As explained in \cite{S-wild} and used in Section \ref{s:6-8}, we choose a prime $p\equiv1\mod 4$ such that Galois action over $\F_p$ 
and automorphisms commute.
Then
the characteristic polynomials on each eigenspace involved can be computed from point counts
over $\F_p, \F_{p^2}$ using Lefschetz fixed point formula, Poincar\'e duality and the Weil conjectures.
At $p=17$, we find the characteristic plynomials
\[
\lambda^4-4i\lambda^3-(16+4i)\lambda^2-(32-60i)\lambda+17(15+8i) \;\; \text{ and }\;\; \lambda-(1-4i)
\]
as well as their conjugates for the eigenspaces in $H^1(C)$ and $H^1(E_i)$.
This gives two possibilities for the characteristic polynomial on $T_2$;
the right choice can be singled out by computing the trace $-32$ of Frob$_{17}^*$ on $T_2$
by point counting directly on the model \eqref{eq:iso-30} of $\mS/\gamma$:
\[
\lambda^8+32\,\lambda^7+816\,\lambda^6+18496\,\lambda^5+383214\,\lambda^4+17^2 18496\,\lambda^3+17^4 816\,\lambda^2+17^6 32\,\lambda+17^8.
\]
Since this characteristic polynomial is irreducible over $\Q$ and visibly not coming from a cyclotomic polynomial,
we infer from \eqref{eq:number} that $T_2$ is irreducible of rank 8 and transcendental,
i.e.~does not contain any algebraic classes over $\bar\F_p$ and thus over $\C$.
Hence the same holds for $T(\mS/\imath_2)$.
By \eqref{eq:T-30}, we obtain 
\[
\rank\, T(\mS) = 7 + 2\cdot 8 = 23
\]
which implies $\rho(\mS) = 30$ as claimed.
\end{proof}

\subsection{Top Picard numbers}
\label{ss:top}

In the remaining sections, we engineer quintics with the top even Picard numbers:
\[
\rho=40,42,44.
\]
To this end, we start with the unique Delsarte quintic
attaining the maximum $\rho=45$  from \cite{S-quintic}:
\[
S:\;\;\; yzw^3+xyz^3+wxy^3+zwx^3=0.
\]
We point out the subgroups of the automorphisms group 
given by cyclic coordinate permutations,
generated by
\[
\varphi: [x,y,z,w] \mapsto [y,z,w,x],
\]
and coordinate multiplication by 15th roots of unity,
generated by
\[
\gamma: [x,y,z,w] = [\zeta\,x, \zeta^3\,y, \zeta^7\,z, w] \;\;\; (\zeta\in\mu_{15}).
\]
Together this gives a subgroup 
of $\Aut(S)$ of size 60. 
The above model has 4 isolated rational double points of type $A_9$.
There are 11 lines:
6 where two homogeneous coordinates vanish simultaneously and the orbit of the line 
\[
\ell=\{x+z=y+w=0\}
\]
under the above automorphisms. 
Adding two (out of 4) twisted cubics, a basis of $\Pic(S)$ is given by the exceptional curves,
three of the first set of lines and 4 of the second set of lines (compare \cite[\S 4]{S-quintic}).

Now we deform $S$ retaining singularities by monomials involving $xyzw$. 
Since the first 6 lines are preserved, this gives $\rho\geq 39$.
For us, it will be more convenient to preserve also the line $\ell$.
Equivalently, the coordinate permutation
\[
\imath= \varphi^2: [x,y,z,w] \mapsto [z,w,x,y]
\]
is preserved. We obtain a 2-dimensional family of quintics
\[
\mS:\;\;\; yzw^3+xyz^3+wxy^3+zwx^3+xyzw(a(x+z)+b(y+w))=0
\]
with $\rho(\mS)\geq 40$. We will see in \ref{ss:40} that generically equality is attained;
that is, up to finite index $\Pic(\mS)=V$ for the lattice generated by the above 40 independent
divisor classes.
Over $\Q$, the involution $\imath$ splits the transcendental Hodge structure into eigenspaces
\begin{eqnarray}
\label{eq:T+-}
T(\mS) = T_+ + T_-.
\end{eqnarray}
From the induced action on $H^{2,0}(\mS)$
we infer that either eigenspace has $T_\pm^{2,0}$ of dimension $2$;
in particular, by complex conjugation
\begin{eqnarray}
\label{eq:T-40}
\rank\,T_\pm \geq 4.
\end{eqnarray}
Of course, the Hodge structure $T_+$ can be interpreted geometrically on the quotient surface $\mY=\mS/\imath$.
In invariant coordinates, we derive the affine elliptic fibration in Weierstrass form
\begin{eqnarray}
\label{eq:quot-40}
\;\;\;\;\;
\mY:\;\;\; y^2 - (t^3+at^2+bt+1) xy + t^3 (t^3+at^2+bt+1) y = x^3 - t^3 x^2.
\end{eqnarray}
Generically, this elliptic surface over $\PP^1$ has reducible fibres of type $I_{12}$ at $t=0,\infty$
and $I_2$ at the zeroes of $t^3+at^2+bt+1$.
At $(0,0)$ there is a 4-torsion section,
so by the Shioda-Tate formula  
\[
\rho(\mY)\geq 27 \;\;\; \text{ and } \;\;\;
\rank\,T_+ = \rank\,T(\mY) \leq 7.
\]
We will see and use below that this agrees with the deformation of $S$.
Consider the orthogonal complement $U$ of $V$ inside $H^2(\mS,\Q)$:
\[
U=V^\perp\subset H^2(\mS,\Q).
\]
Here $U$ has dimension $13$ and contains the transcendental part of $H^2(\mS,\Q)$
by construction.
We will show in \ref{ss:40} that generically $U$ is fully transcendental.
For later use, we determine the dimensions of the eigenspaces $U_\pm$ under $\imath^*$:

\begin{Lemma}
\label{Lem:U}
On any smooth specialisation of $\mS$, we have
\[
\dim_\Q(U_+)=7 \;\; \text{ and } \;\; 
\dim_\Q(U_-)=6.
\]
\end{Lemma}

\begin{proof}
We argue with the smooth specialisation $S$.
Here $U$ contains the 
two four-dimensional eigenspaces  $T(S)_\pm\otimes\Q$
and 5 independent algebraic classes
which we deform away on $\mS$
(at least so we will later show, cf.~Proposition \ref{lem:40}).
In terms of eigenspaces $U_\pm$, we find that
the twisted cubics and  two $\Pic(S)$-dependant sums of lines,
\[
 \ell'+\imath^*\ell', \varphi^*\ell'+(\varphi^*)^3\ell' \;\;\; (\ell'=\gamma^*\ell)
  \]
are $\imath^*$-invariant.
Together these contribute 3 to the rank of $U_+$,  so $\dim_\Q(U_+)\geq 7$.
Meanwhile two differences of lines,
namely
$\ell'-\imath^*\ell', \varphi^*\ell'-(\varphi^*)^3\ell'$
are anti-invariant for $\imath^*$, so $\dim_\Q(U_-)\geq 6$.
Since  $\dim_\Q(U)=13$, we obtain the claimed dimensions on $S$.
Having verified these topological properties on one smooth member of the family $\mS$, 
they necessarily hold for all smooth members.
\end{proof}


\subsubsection{Subfamily}

We shall now specialise further by requiring that the full group $C_4=\langle\varphi\rangle$ of coordinate permutations is preserved:
\[
a=b.
\]
On $\mY$, this does not impact the singular fibers generically, but
it induces the involution $t\mapsto 1/t$ on the base which extends to $\mY$.
Combined with the hyperelliptic involution on the generic fiber,
we obtain two non-trivial involutions on $\mY$.
The quotients are K3 surfaces $\mX_1, \mX_2$ such that over $\Q$
\begin{eqnarray}
\label{eq:T12}
T(\mY) = T(\mX_1) \oplus T(\mX_2).
\end{eqnarray}

\subsubsection{First K3 quotient}
\label{sss:K3-1}

Writing $s=t+1/t$, the quotient preserving the 4-torsion section can be given by
\begin{eqnarray}
\label{eq:X_1}
\mX_1:\;\;\; y^2 = x(x^2+(s+2)(s^3+2bs^2+(b^2+2b-3)s+2b^2-4b+10)x+16(s+2)^2)
\end{eqnarray}
with 4-torsion sections at $x=-4(s+2)$.
The singular fibers of type $I_{12}$ at $s=\infty$, $I_1^*$ at $s=-2$
and $I_2$ at $s=1-b$
show that 
\begin{eqnarray}
\label{eq:T1}
\rho(\mX_1)\geq 19 \;\;\; \text{ and } \;\;\; \rank\,T(\mX_1)\leq 3.
\end{eqnarray}
By the moduli theory of K3 surface over $\C$,
equality is attained generically  in the above rank estimates,
but there are countably many specialisations with $\rho=20$, 
the so-called singular K3 surfaces.
By a result of Livn\'e \cite{Livne}, singular K3 surfaces over $\Q$ are modular.
In fact, the associated Hecke eigenform has CM by an imaginary quadratic field
of class group exponent two.
We employ the techniques from \cite{ES}
to search for K3 surfaces over $\Q$ with $\rho=20$ inside this 1-dimensional family.
Without difficulty we found the following candidates
listed by the square class of the discriminant of the N\'eron-Severi lattice
(or the CM-field of the corresponding modular form of weight 3):
$$
\begin{array}{c||c|c|c|c|c|c|c}
\disc(\Pic)\mod (\Q^\times)^2 & -3 & -8 & -11 & -15 & -51 & -123 & -267\\
\hline
b & 3 & 5 & 7 & 0, 18 & -9 & -45 & -3^3\cdot11
\end{array}
$$
In the sequel (\ref{ss:42}, \ref{ss:44}), we will consider the cases $b=3,5$ in detail
to show that 
$\rho(\mX_1)=20$ and that
the corresponding quintics have Picard number $\rho=42$ resp.~$44$.
Meanwhile we end this paragraph with some comments on the family $\mX_1$
from the moduli point of view.

Using the generators of the N\'eron-Severi lattice,
one can easily compute with the discriminant form following Nikulin 
that generically $\mX_1$ has transcendental lattice
\[
T(\mX_1) = U + \langle 6\rangle.
\]
It follows, that $\mX_1$ arises from a pair of 3-isogenous elliptic curves
through a Shioda--Inose structure.
This suggests that the family should be related to $X^*(3)$,
the modular curve parametrising pairs of conjugate 3-isogenous elliptic curves \cite{Quer}.
More precisely,
we only have to eliminate a symmetry in the parameter $b$ (which can derived 
from easy algebraic manipulations of \eqref{eq:X_1}) 
by extracting a cube root of unity :
\[
c=(b-3)^3.
\]
Then one checks indeed that the parameter $a$ from \cite{Quer} is related to $c$ by
the projective transformation
\[
c=108a/(a-1).
\]
For space reasons we omit the details which proceed along the same lines
as the arguments in \cite{GS}.

\subsubsection{Second K3 quotient}

The second K3 quotient $\mX_2$ is the quadratic twist of $\mX_1$
at the ramified fibers at $s=\pm 2$ of  the involution on $\PP^1$:
\[
\mX_2:\;\;\; y^2 = x(x^2+(s-2)(s^3+2bs^2+(b^2+2b-3)s+2b^2-4b+10)x+16(s-2)^2)
\]
Generically, there are a two-torsion section at $(0,0)$
and reducible fibers of Kodaira types
$I_{12}$ at $s=\infty$, $I_0^*$ at $s=2$
and $I_2$ at $s=1-b$.
By the Shioda-Tate formula, we infer 
\begin{eqnarray}
\label{eq:T2}
\rho(\mX_2)\geq 18 \;\;\; \text{ and } \;\;\; \rank\,T(\mX_1)\leq 4.
\end{eqnarray}
We will see in Lemma \ref{lem:T_2} that generically equality is attained in \eqref{eq:T2}.

\subsubsection{Action on $T_-\subset U_-$}

We have seen how the automorphism $\varphi$ helps us understand the Hodge structures 
$T_+\subset U_+$;
it also has an impact on the Hodge structures $T_-\subset U_-$.
Namely, since $\imath^*=(\varphi^*)^2$ acts as $-1$ on $U_-$
(the Prym part of $U$ with respect to $\imath^*$ so-to-say),
the induced automorphism $\varphi^*$ endows $U_-$
with the structure of a $\Z[i]$-module.
This implies in particular that both $U_-$ and $T_-$ have even rank.
Thus we infer from Lemma \ref{Lem:U} and  \eqref{eq:T-40} that 
\begin{eqnarray}
\label{eq:T-}
\rank\,T_- = 4 \;\; \text{ or } \;\; 6.
\end{eqnarray}
Again, generically the second equality is attained as we will see in Lemma \ref{lem:T_-}.
Moreover the automorphism $\varphi$ proves very useful 
when computing the characteristic polynomial of Frobenius
on the Galois representation associated to $T_-\subset U_-$.
Here we employ techniques like in \ref{ss:spec-30} (see \cite{S-wild}).

\subsection{Picard number 40}
\label{ss:40}

We shall now go through the remaining Picard numbers one by one.
We start by proving that the generic $\mS$ has $\rho=40$.

\begin{Proposition}
\label{lem:40}
At $b=1$, the complex quintic $\mS$ has $\rho(\mS)=40$.
\end{Proposition}

We will prove the equivalent statement that
\[
\rank\,T(\mS) = 13.
\]
Using \eqref{eq:T+-} and \eqref{eq:T12}, the statement can be broken down into
rank computations for the lattices $T_1, T_2, T_-$.
These will be achieved in Lemmata \ref{lem:T_2}, \ref{lem:T_-},  \ref{lem:T_1} below.
Throughout we use reduction modulo some primes.
We start with $T_2$ and $T_-$ because their treatment requires only one prime.

\begin{Lemma}
\label{lem:T_2}
At $b=1$, the K3 surface $\mX_2$ has $\rho(\mX_2)=18$ and $\rank\,T(\mX_2)=4$.
\end{Lemma}

\begin{proof}
In terms of singular fibers and torsion sections,
we know a sublattice $V_2$ of $H^2(\mX_2)$ of rank $18$
generated by algebraic classes.
Here $T(\mX_2)\subseteq U_2= V_2^\perp$.
We claim that equality is attained.
To see this, 
let $p$ be a prime of good reduction
and compute the characteristic polynomial of Frobenius on the Galois representation $U_2\otimes\Q_\ell (\ell\neq p)$.
This can be regarded in $\het{\mX_2\otimes\bar\Q}$,
or via the comparision theorem, in $\het{\mX_2\otimes\bar\F_p}$.
For shortness, we will abbreviate this as 'over $\bar\Q$' (or on $\mX_2\otimes\bar\Q$)
and as 'over $\bar\F_p$'  (or on $\mX_2\otimes\bar\F_p$) in what follows.

Using the Lefschetz fixed point formula and point counting over $\F_p$ and $\F_{p^2}$,
we obtain the traces of $\Fr_p^*$ and $\Fr_{p^2}^*$
on $U_2\otimes\Q_\ell$.
Then the characteristic polynomial $\chi_p(\lambda)$ can be derived from Poincar\'e duality
once we know the sign of the functional equation.
In particular, if $(\tr\Fr_p^*)^2 \neq \tr\Fr_{p^2}^*$,
then the middle coefficient of $\chi_p(\lambda)$ is non-zero,
and the sign of the functional equation is positive, so the two traces determine $\chi_p(\lambda)$. 
At $b=1$ and $p=11$ we thus obtain
\[
\chi_p(\lambda) = \lambda^4+\lambda^3-132\lambda^2+121\lambda+14641.
\]
If $U_2\otimes\Q_\ell$ were to include some algebraic class on $\mX_2\otimes\bar\F_p$,
then this would have eigenvalue $p$ times a root of unity under $\Fr_p^*$.
That is, $\chi_p(\lambda)$ would have a factor coming from a cyclotomic polynomial.
Presently, however, $\chi_p(\lambda)$ is irreducible
and not of cyclotomic shape.
Hence, by \eqref{eq:number}
\[
18\leq \rho(\mX_2) \leq \rho(\mX_2\otimes\bar\F_p) =18,
\]
and the lemma follows.
\end{proof}

We now turn to the lattice $T_-$ on the quintic $\mS$ at $b=1$:

\begin{Lemma}
\label{lem:T_-}
At $b=1$, we have $U_-=T_-\otimes\Q$ of $\Q$-dimension 6.
\end{Lemma}

\begin{proof}
We could proceed as above,
but this would require counting points deeper into finite fields (at least down to $\F_{p^3}$)
than we would like.
Instead, we utilise the automorphism $\varphi$ which commutes with $\Gal(\bar\F_p/\F_p)$
if $p\equiv 1\mod 4$.
Then we can compute the traces of the compositions
$(\varphi^j\circ\Fr_q)^*$ on $U\otimes\Q_\ell$ from points counts over $\F_q$.
This encodes the eigenvalues of $\Fr_q^*$ on all 4 eigenspaces of $\varphi^*$ on $U\otimes\Q_\ell$.
In particular, we obtain the traces on the $(\pm i)$-eigenspaces into which $U_-\otimes\Q_\ell$
decomposes.
For instance, point counts over $\F_p, \F_p^2$ give the first two  coefficients of the characteristic polynomial
of $\Fr_p^*$ on $(U\otimes\Q_\ell)^{\varphi^*=\pm i}$.

We now specialise to the smooth point $b=1$ and $p=5$.
Since
\[
(U\otimes\Q_\ell)^{\varphi^*=i} \oplus (U\otimes\Q_\ell)^{\varphi^*=-i} = U_-\otimes\Q_\ell,
\]
either eigenspace is 3-dimensional by Lemma \ref{Lem:U}.
Thus there is only one coefficient over $\Z[i]$ missing
from the characteristic polynomial.
We can solve for this coefficient by spelling out Poincar\'e duality
for the characteristic polynomial of $\Fr_p^*$ on $U_-\otimes\Q_\ell$
(a system of 3 equations in 2 variables with a solution over $\Z$
compatible with the Weil conjectures).
At $b=1$ and $p=5$,
this leads, uniquely up to conjugation, to the characteristic polynomial
\[
\det(\lambda-{\Fr}_p^*; (U\otimes\Q_\ell)^{\varphi^*=i}) = 
\lambda^3-4(1-i)\lambda^2+20(1-i)\lambda+125i.
\]
In particular, we obtain that the characteristic polynomial of $\Fr_p^*$ on $U_-\otimes\Q_\ell$
is irreducible over $\Q$ and not of cyclotomic shape.
As in the proof of Lemma \ref{lem:T_2},
we deduce that $U_-\otimes\Q_\ell$ contains no algebraic classes 
regarded in $\het{\mX_2\otimes\bar\F_p}$.
Hence $U_-$ cannot contain any algebraic classes in $H^2(\mX_2,\Q)$ either;
equivalently, $U_-=T_-\otimes\Q$ of dimension 6.
\end{proof}

It remains to compute the rank of $T_1$ at $b=1$.

\begin{Lemma}
\label{lem:T_1}
At $b=1$, the K3 surface $\mX_1$ has $\rho(\mX_1)=19$ and $\rank\,T(\mX_1)=3$.
\end{Lemma}

\begin{proof}
Fiber components and sections provide us with a sublattice $V_1\subset\Pic(\mX_1)$ of rank 19.
As before, we denote its orthogonal complement by $U_1\subset H^2(\mX_1,\Z)$.
We continue by computing the characteristic polynomial $\chi_p(\lambda)$ of Frobenius on the Galois representation
$U_1\otimes\Q_\ell$. 
Depending on the sign of the functional equation, $\chi_p(\lambda)$
is determined by the trace of $\Fr_p^*$ on $U_1\otimes\Q_\ell$
which can be obtained from point counting. 
In detail we find at $b=1$
\begin{eqnarray*}
\chi_5(\lambda) & = & (\lambda-5) (\lambda^2+3\lambda+25) \;\;
\text{ or } \;\; (\lambda+5) (\lambda^2-7 \lambda+25)\\
\chi_{11}(\lambda) & = & (\lambda-11) (\lambda^2+13 \lambda+121) \;\;
\text{ or } \;\; (\lambda+11) (\lambda^2-9 \lambda+121)
\end{eqnarray*}

Clearly $T(\mX_1)\subseteq U_1$ (where the embedding is primitive by definition).
If there was no equality, then $U_1$ were to contain an algebraic class over $\bar\Q$
(not only over $\bar\F_p$ as predicted by the Tate conjecture for parity reasons).
That is, $\rho(\mX_1)=20$ and $\mX_1$ would be modular by \cite{Livne}.
In particular, the characteristic polynomial of $\Fr_p^*$ on $T(\mX_1)\otimes\Q_\ell$
would always split over the same imaginary quadratic extension of $\Q$
(in agreement with the Artin-Tate conjecture at the split primes, cf.~\cite{S-NS}).
Presently, the splitting field would have discriminant
\[
d=-51\; \text{ or } \; -91 \;\; \text{ at } \;\; p=5 \;\;\; \text{ and } \;\;\; d=-35 \; \text{ or } \; -403 \;\; \text{ at } \;\; p=11.
\]
Since neither discriminants agree, we infer that $U_1$ does not contain an algebraic class over $\bar\Q$.
Hence $T(\mX_1)=U_1$ of rank 3.
\end{proof}

\begin{proof}[Proof of Proposition \ref{lem:40}]
We have verified that the three transcendental lattices comprising $T(\mS)$ up to finite index
have a total rank of 13 at $b=1$.
Thus $\rho(\mS)=40$ as required.
\end{proof}

\begin{Remark}
It is clear from the proofs that the same arguments go through 
for any rational number $b\equiv 1\mod 55$.
Without difficulty one can determine other congruence classes with $\rho(\mS)=40$ in abundance.
\end{Remark}

\subsection{Picard number 42}
\label{ss:42}

In this subsection,
we specialise to the quintic $\mS$ at $b=3$.

\begin{Proposition}
\label{prop:42}
At $b=3$, the complex quintic $\mS$ has $\rho(\mS)=42$.
\end{Proposition}

Notably the specialisation introduces an additional singularity of type $A_2$ at $[1,-1,1,-1]$
which is $\varphi$-invariant.
For the minimal resolution $\mS$, we infer 
\[
\rho(\mS)\geq 42,
\]
and the proposition claims that equality holds true.
On the quotient surface $\mY$,
the specialisation results in the degeneration of the three singular fibres of type $I_2$
to a single fiber of type $I_6$.
Thus 
\[
\rho(\mY)\geq 29
\]
 by the Shioda-Tate formula. We will prove equality
by analysing the transcendental lattices of the subsequent quotient surfaces 
$\mX_1$ and $\mX_2$.
Here $\mX_1$ is immediate:
its reducible fibers degenerate to $I_{12}, I_3^*$,
so $\mX_1$ is a singular K3 surface of discriminant $-3$
(affirming the entry in the table in \ref{sss:K3-1}).

\begin{Lemma}
\label{lem:T2-42}
At $b=3$, the K3 surface $\mX_2$ has $\rho(\mX_2)=19$ and $\rank\,T(\mX_2)=3$.
\end{Lemma}

\begin{proof}
We proceed as in the proof of Lemma \ref{lem:T_1}.
Namely the torsion sections and the singular fibres of type $I_{12}, I_0^*, I_3$
 generate a sublattice $V_2\subset\Pic(\mX_2)$ of rank 19
 with orthogonal complement $U_2\subset H^2(\mX_2,\Z)$.
We continue by computing the characteristic polynomial $\chi_p(\lambda)$ of Frobenius on the Galois representation
$U_2\otimes\Q_\ell$. 
For the proof of Lemma \ref{lem:T_2},
we  counted points over $\F_p, \F_{p^2}$ throughout the family for small $p$. 
At $b=3$ this gives:
\begin{eqnarray*}
\chi_7(\lambda) & = &(\lambda+7) (\lambda^2-5 \lambda+49),\\
\chi_{11}(\lambda) & = & (\lambda+11) (\lambda^2-4 \lambda+121).
\end{eqnarray*}
Since the splitting fields of the quadratic factors, $\Q(\sqrt{-19})$ and $\Q(\sqrt{-13})$,
 are distinct, we infer as in the proof of Lemma \ref{lem:T_1} that $\mX_2$ cannot have Picard number 20 over $\bar\Q$ at $b=3$.
Hence $T(\mX_2)=U_2$ of rank 3.
\end{proof}

To complete the proof of Proposition \ref{prop:42},
it remains to compute the rank of $T_-$ at $b=3$:

\begin{Lemma}
\label{lem:T-42}
At $b=3$, we have $U_-=T_-\otimes\Q$ of $\Q$-dimension 6.
\end{Lemma}

\begin{proof}
In principle, we proceed as in the proof of Lemma \ref{lem:T_-},
but there is the little subtlety that we have to determine the dimension of $U_-$ first
since the specialisation in the family $\mS$ is not smooth.
This can be dealt with as follows.
At any rate we have 
\[
\dim_\Q U_-=4
\;\; \text{ or } \;\; 6
\]
by the $\Z[i]$-module structure (cf.~\eqref{eq:T-}).
For starters, we compute the traces of $\Fr_p^*, \Fr_{p^2}^*$ on the eigenspaces
$(U\otimes\Q_\ell)^{\varphi^*=\pm i}$
through point counts over  $\F_p, \F_p^2$.
Up to conjugation, we find at $p=13$
\[
\tr{\Fr}_p^*((U\otimes\Q_\ell)^{\varphi^*=i}) = 
20+12i,\;\;\; 
\tr{\Fr}_{p^2}^*((U\otimes\Q_\ell)^{\varphi^*=i})  =
-56+40i.
\]
If $U_-$ were 4-dimensional, then these two traces would determine
the characteristic polynomial of $\Fr_p^*$ on $U_-\otimes\Q_\ell$;
presently the resulting quartic polynomial does not satisfy Poincar\'e duality as the constant coefficient is even and thus differs from $p^4$.
Thus $\dim_\Q(U_-)=6$ and we proceed 
by solving for the constant coefficient (over $\Z[i]$) 
of the characteristic polynomial using Poincar\'e duality.
Uniquely up to conjugation, this leads to
\[
\det(\lambda-{\Fr}_p^*; (U\otimes\Q_\ell)^{\varphi^*=i}) = 
\lambda^3+4(5+3i)\lambda^2+52(3+5i)\lambda+13^3i.
\]
In particular, we obtain that the characteristic polynomial of $\Fr_p^*$ on $U_-\otimes\Q_\ell$
is irreducible over $\Q$ and not of cyclotomic shape.
Hence $U_-\otimes\Q_\ell$ does not contains any algebraic classes 
regarded in $\het{\mX_2\otimes\bar\F_p}$,
so neither in $\het{\mX_2\otimes\bar\Q}$.
That is, $U_-=T_-\otimes\Q$ of dimension 6.
\end{proof}

\begin{proof}[Proof of Proposition \ref{prop:42}]
Recall from \eqref{eq:T+-}, \eqref{eq:T12} that the three transcendental lattices
$T(\mX_1), T(\mX_2)$ and $T_-$ 
comprise $T(\mS)$ up to finite index.
We have seen that they
have a total rank of $11=2+3+6$ at $b=3$.
Thus $\rho(\mS)=42$ as required.
\end{proof}

\subsection{Picard number 44}
\label{ss:44}

Finally we consider the specialisation of $\mS$ at $b=5$.
This comes with two additional $A_2$ singularities at $[1,\alpha,-1,-\alpha]$ for $\alpha=1\pm\sqrt 2$.
Hence $\rho(\mS)\geq 44$,
i.e.~at most one sub-Hodge structure $T_\bullet \,(\bullet =1, 2, -)$ could have rank exceeding
$2\dim T^{2,0}_\bullet$ (by one).
We claim:

\begin{Proposition}
\label{prop:44}
At $b=5$, the complex quintic $\mS$ has $\rho(\mS)=44$.
\end{Proposition}

To prove the proposition, we first note that the additional $A_2$ singularities are $\imath$-invariant.
In fact, each exceptional curve is preserved by $\imath$.
Hence, these classes are absorbed by $U_+$,
so $\rank\,U_+\geq \rank\,T_++4\geq 8$ by \eqref{eq:T-40}.
This has two consequences.
On the one hand, $U_-$ is degenerate
as $\dim_\Q U_- = \dim_\Q U - \dim_\Q U_+\leq 13-8=5$.
More precisely, by \eqref{eq:T-40} and \eqref{eq:T-} we have automatically 
\begin{eqnarray}
\label{eq:T-44}
\rank\,T_- = 4 \;\;\; \text{ and } \;\;\; U_-=T_-\otimes\Q.
\end{eqnarray}
As a sanity check we computed the characteristic polynomial of $\Fr_p^*$ on $T_-\otimes\Q_\ell$.
Proceeding as in the hypothetical degenerate case of the proof of Lemma \ref{lem:T-42},
we found $\chi_{13}(\lambda)=\lambda^4+6 \lambda^3+18 \lambda^2+1014 \lambda+28561$.

On the other hand, there are more $\imath^*$-invariant algebraic classes in $H^2(\mS,\Z)$
than possibly accommodated by the quotient elliptic surface.
Indeed, the resolution of the quotient $\mS/\imath$ has $\rho\geq 31$,
but it also contains two $(-1)$-curves whose blow-down yields the
Kodaira-N\'eron model $\mY$ of \eqref{eq:quot-40} at $b=5$ with $\rho(\mY)\geq 29$.
We continue by analysing the two K3 quotients of $\mX$.
As in \ref{ss:42},
$\mX_1$ defines a singular K3 surface at $b=5$;
the sections of height $4/3$ at $x=4$ show that 
$\Pic(\mX_1)$ has rank 20 and discriminant $-8$,
affirming the respective entry in the table in \ref{sss:K3-1}.
Hence 
\begin{eqnarray}
\label{eq:T1-44}
\rank\,T(\mX_1)=2.
\end{eqnarray}
For $\mX_2$, this implies $\rank\,T(\mX_2)=2$ or $3$,
and we claim the latter.
Indeed, $\rho(\mX_2)\geq 19$, but we did not try to determine the additional section
compared to the generic member of the family $\mX_2$.
Instead, we simply argue with the rank 18 sublattice $V_2\subset\Pic(\mX_2)$ 
generated by torsion sections and fiber components (as in the proof of Lemma \ref{lem:T_2}).
Then its orthogonal complement $U_2=V_2^\perp\subset H^2(\mX_2,\Z)$
surely contains an algebraic class,
but we have to rule out that it contains two independent ones.
We proceed by determining the characteristic polynomial $\chi_p(\lambda)$
of Frobenius on $U_2\otimes\Q_\ell$
by counting points over $\F_p, \F_{p^2}$.
At $p=7, 11$ we obtain
\begin{eqnarray*}
\chi_7(\lambda)  & = & (\lambda-7) (\lambda+7) (\lambda^2+4 \lambda+49), \\
\chi_{11}(\lambda)  & = &  (\lambda-11) (\lambda+11) (\lambda^2+10 \lambda+121).
\end{eqnarray*}
Since the splitting fields of the quadratic factors,
$\Q(\sqrt{-5})$ and $\Q(\sqrt{-6})$,
do not agree, we infer as before that $\mX_2$ cannot be modular and hence $\rho(\mX_2)<20$.
In conclusion, 
\begin{eqnarray}
\label{eq:T2-44}
\rho(\mX_2)=19 \;\;\; \text{ and } \;\;\;\rank\,T(\mX_2)=3.
\end{eqnarray}
Summing up \eqref{eq:T-44} - \eqref{eq:T2-44} gives $\rank\,T(\mS)=9$ at $b=5$ by \eqref{eq:T-40} and \eqref{eq:T12}.
Hence $\rho(\mS)=44$ as claimed.
\qed

\subsection*{Acknowledgements}

I thank Ronald van Luijk for the initial discussions in 2010 concerning the first couple 
of cases of this project, and in particular, concerning K3 surfaces of Picard number one.
Part of these discussions took place 
during the Junior Trimester 'Algebra \& Number Theory'
at the Hausdorff Institute of Mathematics
which I thank for the generous 
hospitality.
I am grateful to the referee for the helpful comments.

\end{document}